\documentclass[11pt]{amsart}
\usepackage{geometry}
\numberwithin{equation}{section}
\usepackage{amssymb}
\usepackage{xcolor}
\usepackage[T1,T5]{fontenc}
\usepackage{hyperref}
\hypersetup{   
colorlinks=true,       
linkcolor=black,          
citecolor=black,        
filecolor=black,      
urlcolor=black}     
\newcommand{\etalchar}[1]{$^{#1}$}

\newtheorem{theorem}{Theorem}[section]
\newtheorem{lemma}[theorem]{Lemma}
\newtheorem{lem}[theorem]{Lemma}
\newtheorem{prop}[theorem]{Proposition}
\newtheorem{proposition}[theorem]{Proposition}
\newtheorem{corollary}[theorem]{Corollary}

\theoremstyle{definition}

\newtheorem{definition}[theorem]{Definition}
\newtheorem{example}[theorem]{Example}

 \newtheorem*{ackn}{Acknowledgements}
 \newtheorem*{org}{Organization}

 \newcommand{\B}{\mathbb B}
 
 \newcommand{\R}{\mathbb R}
 
 \newcommand{\C}{\mathbb C}
 \newcommand{\N}{\mathbb N}

\begin{document}

\title[Complex Monge-Amp\`ere equations]{Complex Monge-Amp\`ere equations  with \\ boundary values:  A variational approach}

  \author{Papa Badiane}
\address{Papa Badiane: Laboratoire de Mathématiques et Applications\\ UFR des Sciences et Technologies\\
Université Assane Seck de Ziguinchor (Sénégal) }

\email{\href{mailto:p.badiane4963@zig.univ.sn }{\href{mailto:p.badiane4963@zig.univ.sn}{p.badiane4963@zig.univ.sn}}}

\author{Chinh H. Lu}
\address{Institut Universitaire de France \& Univ Angers, CNRS, LAREMA, SFR MATHSTIC, F-49000 Angers, France}
\email{\href{hoangchinh.lu@univ-angers.fr}{hoangchinh.lu@univ-angers.fr}} 

\author{Souhaibou Sambou}
\address{Laboratoire d'Analyse Numérique et Informatique\\ UFR des Sciences Appliquées et Technologies\\
	Université Gaston Berger de Saint-Louis (Sénégal)} 
\email{\href{souhaibou.sambou@ugb.edu.sn}{souhaibou.sambou@ugb.edu.sn}}

\author{Ahmed Zeriahi}
\address{Ahmed Zeriahi: Institut de Math\'ematiques de Toulouse   \\ Universit\'e de Toulouse \\
118 route de Narbonne,  31400 Toulouse, France.}

\email{\href{mailto:ahmed.zeriahi@math.univ-toulouse.fr}{ahmed.zeriahi@math.univ-toulouse.fr}}



\begin{abstract} Let $\Omega \Subset \mathbb{C}^n$ be a bounded hyperconvex domain and $\phi : \Omega \rightarrow \mathbb{R}$ a maximal plurisubharmonic function  in $\Omega$ (possibly unbounded).
We  introduce a new class $\mathcal E^1_\phi(\Omega)$ of plurisubharmonic functions with  boundary values equal to $\phi$ in the sense of Cegrell and finite $\phi$-energy on  $ \Omega$. Then given a Borel measure $\mu$  on $\Omega$, we apply a variational approach  to prove that $\mu$ is the Monge-Ampère measure of a function $u \in \mathcal E^1_\phi(\Omega)$ if and only if it satisfies the integrability condition $\mathcal E^1_\phi(\Omega) \subset \phi+ L^1(\Omega,\mu)$.	We also establish a quantitative version of this condition extending previous results of Cegrell and others. Using a logarithmic transformation, we obtain analogous results for the real Monge-Amp\`ere equations with (unbounded) boundary data. 
\end{abstract}

\maketitle

\tableofcontents

\section{Introduction}
Given a bounded domain $\Omega \Subset \mathbb C^n$, boundary data $\phi$, and a positive Borel measure $\mu$ on $\Omega$ with locally finite mass, the Dirichlet problem for the complex Monge--Ampère operator with boundary value $\phi$ and right-hand side $\mu$ consists of finding a plurisubharmonic function $u$ such that
 \begin{equation}\label{eq:MA}
\left\lbrace
\begin{array}{lcr}
(dd^c u)^n= \mu & \textnormal{in} & \Omega,\\
u=\phi& \textnormal{on} & \partial \Omega.
\end{array}
\right.
\end{equation} 

When the data are sufficiently regular, one seeks a smooth solution $u$. In this case, the Monge--Ampère measure $(dd^c u)^n$ is, up to a positive multiplicative constant, the determinant of the complex Hessian of $u$. For degenerate data, however, one must consider weaker notions of solution. A fundamental result of Bedford and Taylor \cite{BT76} shows that the Monge--Ampère measure can be defined as a positive Radon measure for locally bounded plurisubharmonic functions.

When $\Omega$ is strictly pseudoconvex, $\phi\in C^0(\partial\Omega)$, and $\mu=fdV$ with $0\leq f\in C^0(\overline{\Omega})$, Bedford and Taylor \cite{BT76} proved that \eqref{eq:MA} admits a unique solution $u\in {\rm PSH}(\Omega)\cap C^0(\overline{\Omega})$. Later, Cegrell \cite{Ceg84} treated the case $f\in L^\infty(\Omega)$ and showed that the solution belongs to ${\rm PSH}(\Omega)\cap L^\infty(\Omega)$. Subsequently, Ko\l odziej \cite{Kol98} established existence and uniqueness of a solution $u\in {\rm PSH}(\Omega)\cap C^0(\overline{\Omega})$ when $f\in L^p(\Omega)$ for some $p>1$. 

On the other hand, when $n = 1$ the complex Monge-Ampère operator reduces to the Laplace operator  $dd^c u =   \Delta u$ and \eqref{eq:MA} is the classical Dirichlet problem for the Laplace operator.
When $\mu$ is singular, already in one dimension a solution--if it exists--may fail to attain the prescribed boundary values $\phi$ in the strong sense. A classical way to address this issue is to formulate a generalized Dirichlet problem using the notion of the least harmonic majorant. Given a harmonic function $h$ on $\Omega$, one seeks a subharmonic function $u$ such that $dd^c u = \mu$ and the least harmonic majorant of $u$ is $h$. The solution is then unique and, for sufficiently regular domains, is represented by the Riesz formula. Moreover, the boundary value is partially attained, $\limsup_{z\to \zeta} (u-h)(z) =0,\; \zeta\in \partial \Omega$. 
\smallskip

In higher dimension the problem is much more complicated because the complex Monge-Ampère operator is nonlinear and it is not defined for all plurisubharmonic functions. In his foundational works \cite{Ceg98,Ceg04}, Cegrell  introduced several classes of plurisubharmonic functions for which the Monge-Amp\`ere operator can be conveniently defined and enjoys natural continuity properties.  In these works the boundary data are assumed to be zero, $\phi\equiv 0$.  For nonzero boundary data that extend to a negative maximal plurisubharmonic function on $\overline{\Omega}$, Cegrell introduced the concept of least maximal plurisubharmonic (LMP) majorant \cite{Ceg08}, and partial results were subsequently obtained in Cegrell and several other authors; see \cite{Ceg08}, \cite{AC07}, \cite{ACCP09}.  More precisely we say that a plurisubharmonic function $u \in \mathcal E(\Omega)$ has boundary values $0$ if its least maximal plurisubharmonic (LMP) majorant $ \widetilde u$ is identically $0$ in $\Omega$ i.e. $\widetilde u \equiv 0$. We denote this class by $\mathcal N_0 (\Omega)$. Given a maximal plurisubharmonic function $\phi$ on $\Omega$ we say that $u$ has boundary value $\phi$ and write $u \in \mathcal N_\phi(\Omega)$ if there is a function $\psi \in \mathcal N_0(\Omega)$ such that $\psi + \phi \leq u \leq \phi$ in $\Omega$. This implies that $\widetilde u = \phi$ but it is a more precise way to say that the function $u-\phi$ has zero boundary value, while it's not plurisubharmonic. We will write this condition as $\widetilde{u-\phi} \equiv 0$. 

The purpose of the present work is to investigate the generalized Dirichlet problem in a bounded hyperconvex domain when the prescribed boundary value $\phi$ is a maximal plurisubharmonic function belonging to the domain of definition of the complex Monge--Ampère operator: 

\begin{equation}\label{eq:MA2}
\left\lbrace
\begin{array}{lcr}
(dd^c u)^n= \mu & \textnormal{on} & \Omega,\\
\widetilde{u-\phi} \equiv 0 & \textnormal{in} &  \Omega.
\end{array}
\right.
\end{equation} 

To this end, we introduce the class $\mathcal E^1_{\phi}({\Omega})$ of plurisubharmonic functions with boundary value $\phi$ and finite $\phi$-energy:
\begin{itemize}
\item there exists  $v \in  \mathcal E^1 (\Omega)$ such that $v + \phi \leq u \leq \phi$ in $\Omega$;
\item the partial $\phi$-energy of $u$ is finite i.e. $e_\phi(u) := \int_\Omega (\phi - u) (dd^c u)^n < \infty$. 
\end{itemize}
Observe that the first condition expresses that the LMP majorant of $u-\phi$ is identically $0$ in $\Omega$; see Subsection \ref{sect: MA operator}. When $n = 1$, by the Riesz representation formula one can show that $\mathcal{E}^1_{\phi}(\Omega)=\phi+\mathcal E^1(\Omega)$. This is no longer the case in higher dimensional; as demonstrated in Example \ref{example: E1}, even the inclusion $\phi+ \mathcal E^1(\Omega) \subset \mathcal{E}^1_{\phi}(\Omega)$  fails in general. 

A key ingredient in our approach is the integration by parts formula established in Theorem \ref{thm: integration by parts}. This formula enables us to derive several fundamental properties of the (full) energy functional $E_{\phi}$, including convexity and monotonicity. These properties, in turn, allow us to apply the variational framework developed in \cite{BBGZ13} and thereby solve the generalized Dirichlet problem.
 
 \begin{theorem} \label{thmA}
 Let $\mu$ be a non pluripolar positive Borel measure on $\Omega$ and $\phi \in  \mathcal M (\Omega)$. 
  
  Then the following properties are equivalent :
 \begin{enumerate}
 \item $\mathcal E^1_{\phi}({\Omega}) \subset \phi + L^1(\Omega,\mu);$
 \item there exists a constant $C = C(\mu,\Omega) > 0$ such that for any $v \in \mathcal E^1_{\phi}({\Omega})$,
 $$
 \int_\Omega (\phi - v) d \mu \leq C \left( 1 +  e_\phi(v)^{1\slash 2}\right);
 $$
 \item there exists a unique $u \in \mathcal{E}_\phi^1(\Omega)$ such that $(dd^c u)^n = \mu$.
 \end{enumerate}
\end{theorem}

Here, $\mathcal M(\Omega)$ consists of all maximal plurisubharmonic functions that are in the Cegrell class $\mathcal E(\Omega)$ (see Definition \ref{def: max psh}). 

The uniqueness part in (3) was known by \cite[Theorem 3.2]{ACCP09},  since $\mathcal E^1_{\phi}(\Omega) \subset \mathcal N_{\phi}(\Omega)$. Observe that the solution provided by the previous theorem satisfies the "boundary condition" $\widetilde u = \phi$. By elementary reason, one sees that the condition $\widetilde{u-\phi}=0$ implies $\tilde{u}=\phi$. The reverse direction is still an open question; we refer to \cite[Theorem 5.1]{ACLR25} for a partial positive answer. 
We also emphasize that, in our main result above, the maximal function $\phi$ can be unbounded from above, while in previous works it was assumed to be negative.

It turns out that when $\phi \equiv 0$, the class $\mathcal E^1_{0}(\Omega)$ coincides with the class $\mathcal E^1(\Omega)$ introduced earlier by Cegrell in \cite{Ceg98}. We also note that the variational approach has previously been used in the local setting to provide alternative proofs of Cegrell's original results; see \cite{ACC12,Lu15}. In those works, however, the boundary data are identically zero, making the problem homogeneous. Our theorem therefore significantly extends the main results of \cite{Ceg98}, \cite{ACC12}.

The following consequence provides many important examples of measures to which our theorem can be applied.

\begin{corollary} \label{coroB} Let $\mu$ be a positive Borel measure on $\Omega$. Assume that there exists 
$v \in \mathcal E_0^1(\Omega)$ such that $\mu \leq (dd^c v)^n$ on $\Omega$. Then for any $\phi \in \mathcal M (\Omega)$ there exists a unique $u \in \mathcal E^1_\phi(\Omega)$ such that $v  + \phi \leq  u \leq \phi$ and
$(dd^c u)^n = \mu$ on $\Omega$.
\end{corollary}
This result says that if  the Dirichlet problem \eqref{eq:MA2} with boundary value $0$ admits a subsolution, then it admits a solution for any boundary value. As shown in Example \ref{example: finite entropy} any probability measure on $\Omega$ with finite entropy satisfies the above assumption. 
\smallskip

We conclude this introduction with a direct application of our main result to real Monge--Ampère equations. The starting point is a deep connection between complex and real Monge--Ampère equations, discovered in \cite{BB13}, \cite{CGSZ19} and further developed in \cite{LZ25}. As mentioned above, functions in the Cegrell class $\mathcal N_0(\Omega)$ have zero boundary values only in a generalized sense, rather than in the strong sense. However, \cite[Proposition 7.3]{LZ25} shows that, under the additional assumption of $(\mathbb S^1)^n$-symmetry, the boundary limit is in fact zero in the strong sense. Combining this observation with the real--complex correspondence allows us to translate our main result into the framework of real Monge--Amp\`ere equations.

\begin{theorem} \label{thmC} Let $D \Subset \R^n$ be a bounded convex domain, $\varphi$ a  maximal convex function in $D$ and $\nu$ a positive Borel measure on $D$. 

Then the following properties are equivalent:
 \begin{enumerate}
 \item $\mathcal E^1_{\varphi}({D}) \subset \varphi + L^1(D,\nu);$
 \item there exists a constant $C = C(\nu,D) > 0$ such that for any $w \in \mathcal E^1_{\varphi}({D})$,
 $$
 \int_D (\varphi - w) d \nu \leq C \left( 1 +  e_\varphi(w)^{1\slash 2}\right);
 $$
 \item there exists $v \in \mathcal{E}_\varphi^1(D)$ such that ${\rm M}_\R (v) = \nu$.
 \end{enumerate}
\end{theorem}
Here ${\rm M}_{\mathbb R}$ is the real Monge-Amp\`ere operator as discovered by Alexandrov; see \cite{Fig17}, \cite{AZ67}. By analogy with the complex setting, $\mathcal{E}_\varphi^1(D)$ is the set of convex functions in $D$ with boundary value $\varphi$ in {\it the strong sense} that $\lim_{z\to \partial D} (u-\varphi)(z) =0$, and having finite energy $e_\varphi(u) := \int_D (\varphi-u) {\rm M}_\R(u) < + \infty$.  

When $\nu$ has finite total mass, $\nu(D)<+\infty$, and $\phi$ is continuous on $\overline D$, the result is classical in the literature on real Monge--Amp\`ere equations; see \cite[Theorem 1.1]{Har06}, \cite[Theorem 4.1]{RT77}, \cite[Theorem 2.13]{Fig17}, and \cite[Theorem 1.6.2]{Gut16}. We will show that Theorem \ref{thmC} is a reformulation of Theorem \ref{thmA} in the real setting.  Alternatively, one may adapt the variational method directly to the real setting; we refer to the recent work \cite{Le26}, where such an approach is developed, inspired by \cite{LZ25}. We stress, however, that in our theorem the boundary datum $\phi$ is not assumed to be bounded from above. In particular, $\phi$ may attain the value $+\infty$ at certain boundary points. 

\begin{org}
We recall the necessary background on Cegrell's classes in Section \ref{sect: preliminaries}. In Section \ref{sect: new energy class}, we introduce and investigate a new finite energy class, establishing the monotonicity and convexity of the associated energy functional through a key integration by parts formula. In Section \ref{sect: variational approach}, we adapt the variational framework of \cite{BBGZ13} to our setting. This framework is then used in Section \ref{sect: proof of the main results} to prove the main results of the paper.
\end{org}

\begin{ackn}
This work is partially supported by Institut Universitaire de France
and the KRIS project of fondation Charles Defforey. C.H.Lu is partially supported by the Centre Henri Lebesgue ANR-11-LABX-0020-01. There is no use of AI in our work. 
\end{ackn}

\section{Preliminaries}\label{sect: preliminaries}
In the whole paper $\Omega$ will be a fixed bounded hyperconvex domain in $\mathbb C^n$ i.e. it admits a continuous plurisubharmonic exhaustion $\rho : \Omega \longrightarrow ]-\infty, 0[$ which can be always assumed to be strictly plurisubharmonic in $\Omega$ and $-1 \leq \rho <0$ in $\Omega$ (see \cite[Proposition 1.2]{KR81}).

\subsection{The complex Monge-Ampère operator} \label{sect: MA operator}
\subsubsection{The Bedford-Taylor theory} 
In their seminal papers \cite{BT76, BT82}, E. Bedford and B.A. Taylor were able to define the complex Monge-Ampère operator on the positive  convex cone ${\rm PSH}(\Omega) \cap L^\infty(\Omega)$ of bounded plurisubharmonic functions and 
proved its main properties : among them a general comparison principle and the continuity under monotone convergence of uniformly bounded sequences of plurisubharmonic functions (see \cite{GZbook}).

\subsubsection{The Cegrell Theory}\label{subsect: Cegrell theory}
  Later U. Cegrell \cite{Ceg98} was able to extend the complex Monge-Ampère operator to some classes of singular plurisubharmonic functions in  $\Omega \Subset \C^n$.
  
  Following \cite{Ceg98}, we define the set of "pluripotential test functions" as
 $$
\mathcal{E}^0(\Omega)=\{u \in {\rm PSH}(\Omega)\cap L^\infty(\Omega): \,u=0 ~~ \text{on}~~\partial \Omega~~ \text{and} ~\int_{\Omega}(dd^cu)^n<+\infty \}.
$$
Here the notation $u=0 ~~ \text{on}~~\partial \Omega$ means that for any $\zeta \in \partial \Omega,$
$\lim_{z \to \zeta} u(z) = 0$.

Let $\mathcal{F}(\Omega)$ be the class of negative plurisubharmonic functions $u \in {\rm PSH}(\Omega)^-$ such that there exists a non-increasing  sequence $(u_j) \in \mathcal E^0(\Omega)$ which converges to $u$ and satisfies 
$$
\sup_j \int_\Omega (dd^c u_j)^n < + \infty.
$$
Then Cegrell defined $\mathcal{E}(\Omega)$ as the class of functions $u \in {\rm PSH}(\Omega)$ such that for any open set $ \Omega' \Subset \Omega$ there exists $v \in \mathcal F(\Omega)$ such that $u = v$ in $\Omega'$. 
He proved that the complex Monge-Ampère operator is well defind on $\mathcal{E}(\Omega)$ and  is continuous under monotone convergence of sequences in $\mathcal{E}(\Omega)$. 
As part of the definition, Cegrell only considered negative plurisubharmonic functions. 

If $u$ is psh and bounded from above in $\Omega$, then we say that $u\in \mathcal E(\Omega)$ if $u-\sup_{\Omega} u$ belongs to this class in the sense of Cegrell. As shown by B{\l}ocki, the class $\mathcal E(\Omega)$ is of local nature, i.e. a function $u$ belongs to $\mathcal E(\Omega)$ if and only if its restriction on each open subset $\Omega'\subset \Omega$ belongs to $\mathcal E(\Omega')$. In this paper we will also consider psh functions that are not bounded from above. We thus extend the definition of Cegrell so as to cover these functions. 

\begin{definition}
	Let $u$ be a psh function in $\Omega$. We say that $u\in \mathcal E(\Omega)$ if for each relatively compact open subset $\Omega'\Subset \Omega$, $u\in \mathcal E(\Omega')$. 
\end{definition} 
The definition makes sense because on compact sets $u$ is bounded from above by upper semicontinuity.

We let $\mathcal E^a(\Omega)$ denote the set of $u\in \mathcal E(\Omega)$ such that $(dd^c u)^n$ vanishes on pluripolar sets.

The class $\mathcal E^1(\Omega)$ consists of negative psh functions $u$ realized as the limit of a decreasing sequence $(u_j)$ in $\mathcal E^0(\Omega)$ with uniformly bounded energy i.e. $\sup_j \int_{\Omega}(-u_j)(dd^c u_j)^n <+\infty$.

The following fundamental version of the maximum principle was proved in \cite[Theorem 2.2]{BGZ09} (see also \cite{NP09}): 
\begin{theorem} \label{thm:principe du max} Let $u, v \in \mathcal E(\Omega)$. Then
$$
{\bf 1}_{\{u> v\}} (dd^c \sup\{u,v\})^n = {\bf 1}_{\{u> v\}} (dd^c u)^n,
$$
in the sense of Borel measures on $\Omega$.
Moreover, if $u,v\in \mathcal E^a(\Omega)$ then
$$
(dd^c \sup\{u,v\})^n \geq {\bf 1}_{\{u> v\}} (dd^c u)^n + {\bf 1}_{\{u\leq v\}} (dd^c v)^n.
$$
In particular if $v \leq u$ in $\Omega$ and $u,v\in \mathcal E^a(\Omega)$ then 
$$
{\bf 1}_{\{u= v\}} (dd^c v)^n \leq {\bf 1}_{\{u= v\}} (dd^c u)^n.
$$
\end{theorem}
\subsubsection{Least maximal plurisubharmonic majorant}
The concept of least harmonic majorant is classical in Potential Theory (see \cite{Ran95}).

The corresponding concept, called the least maximal plurisubharmonic  (LMP) majorant, was introduced in Pluripotential Theory by Cegrell (\cite{Ceg04, Ceg08}).

A plurisubharmonic function $\phi : \Omega  \longrightarrow \R$ is said to be maximal in $\Omega$ if for any subdomain $D \Subset \Omega$, any plurisubharmonic function $v \in {\rm PSH}(D)$ such that $v \leq \phi$ in $\partial D$ satisfies the inequality $v \leq \phi$ in $D$.

Let $u \in {\rm PSH} (\Omega)$ and assume that $u $ has a maximal plurisubharmonic majorant $\phi$ in $\Omega$ i.e. $u \leq \phi$. 
Let $(\Omega_j)_{j \in \N}$ a fundamental sequence of strictly pseudoconvex domains with smooth boundary in $\Omega$. For each $j \in \N$ define
$$
\widetilde u^j := \sup \{ v \in {\rm PSH}(\Omega) ; \, v \leq u \, \, \text{in} \, \, \Omega \setminus \Omega_j\}.
$$
For any $j \in \N$, $\widetilde u^j  \in {\rm PSH}(\Omega)$ and  $ u \leq \widetilde u^j  \leq \widetilde u^{j+1} \leq \phi$ in $\Omega$. Then $(\widetilde u^j)_{j \in \N}$ is a non decreasing sequence of plurisubharmonic function in $\Omega$. Let $\widetilde u := (\lim_{j \to + \infty} \widetilde u^j)^*$ the upper semi-continuous regularization of its limit. Then $\widetilde u$ is plurisubharmonic and $\widetilde u \leq \phi$ in $\Omega$. Moreover since $\widetilde u^j $ is maximal in $\Omega_j$, it follows that $\widetilde u$ is a maximal plurisubharmonic majorant of $u$ in $\Omega$ which satisfies the inequality $\widetilde u \leq \phi$ in $\Omega$. 
 This means that $\widetilde u$ is the least maximal plurisubharmonic majorant (LMP) majorant of $u$ in $\Omega$.  Moreover, as shown in \cite{Ceg08}, when $u \in  \mathcal E(\Omega)$ then $ \widetilde u \in  \mathcal E(\Omega)$ and $(dd^c \widetilde u)^n = 0$ in $\Omega$.
 \begin{definition}
 	We define  $\mathcal N_0 (\Omega)$ as the set of $u \in \mathcal E(\Omega)$ such that $\widetilde u \equiv 0$.
 \end{definition}
 The condition $\tilde{u}=0$ means that  $u$ has zero boundary values in a generalized sense. By \cite{Ceg08}, \cite{Ceg98}, $\mathcal{E}^1(\Omega) \subset \mathcal N_0(\Omega)$, and $\mathcal F(\Omega) \subset \mathcal N_0(\Omega)$. This is because the Monge-Amp\`ere measure of $\tilde{u}$ is equal to $0$, and by uniqueness $\tilde{u}=0$.   

\begin{definition}\label{def: max psh}
We let $\mathcal M(\Omega)$ denote the set of functions $u \in \mathcal E(\Omega)$ such that $(dd^c u)^n = 0$ in $\Omega$.	
\end{definition}

%Observe that if  $u$ can be extended as a continuous function at the boundary of $\Omega$ then $\widetilde  u = U_{\Omega,\varphi}$ is the Perron-Bremermann envelope of $\varphi := u_{\mid \partial \Omega}$ and $\widetilde  u = f$ in $\partial \Omega$. 

%When $\varphi \in C^0(\partial \Omega)$ extends as a plurisubhramonic function in $\Omega$, then $u \in \mathcal E(\Omega)$ is a solution to the Dirichlet problem \eqref{eq:MA} if and only if $(dd^c u)^n = \mu$ and  $\widetilde u = U_{\Omega,\varphi}$, where $U_{\Omega,\varphi}$ is the Perron-Bremermann envelope of $\varphi$. This is the case when $\Omega$ is a strictly pseudoconvex domain, as shown in \cite{BT76}.

Given a maximal plurisubharmonic function $\phi \in \mathcal M(\Omega)$, we define 
$\mathcal N_\phi (\Omega)$ to be the class of $u \in \mathcal E(\Omega)$ such that there exists $ v \in \mathcal N_0(\Omega)$ such that $v + \phi \leq u \leq \phi$ in $\Omega$. This means somehow that the LMP majorant of $u-\phi$ is $0$ in the sense that $P(u-\phi) \in \mathcal N(\Omega)$. Observe that it implies that the LMP majorant of $u$ coincides with $\phi$ i.e. $\widetilde u \equiv \phi$. 

We let $\mathcal N^a_{\phi}(\Omega)$ denote the set of $u\in \mathcal N_{\phi}(\Omega)$ such that $(dd^c u)^n$ vanishes on pluripolar sets. 

\subsubsection{Stokes formula}
\begin{lemma} \label{lem: conservation mass}
 Let $u,  v, w_1, \cdots, w_{n-p} \in \mathcal{E}(\Omega)$. Assume that $\Omega' \Subset \Omega'' \Subset \Omega$ are two bounded open sets such that $u =v$ in $\Omega \setminus \Omega'$. 
 Then 
 $$
  \int_{\Omega''} (dd^c u)^{p} \wedge T = \int_{\Omega''} (dd^c v)^{p} \wedge T,
 $$
 where  $T= dd^cw_1\wedge \cdots \wedge dd^cw_{n-p}$ if $1 \leq p \leq n-1$ and $T= {\bf 1}_{\Omega}$ when $p = n$.
 \end{lemma}
\begin{proof}
 Let $u_t = \max(u,-t)$, $w_{j,t} = \max(w_j,-t)$ and $v_t=  \max(v,-t)$ for $t > 0$. Since $\Omega' \Subset \Omega'' \Subset \Omega$, there exists a sequence of bounded open sets $(\Omega_j)$ such that for any $j \in \N$, $\Omega' \Subset \Omega_j \Subset \Omega''$ and $\Omega_j \nearrow \Omega''$. 
 Then since $u_t = v_t$ in $\Omega \setminus \Omega'$ by Stokes formula for bounded functions, we have 
 $$
 \int_{\Omega_j} (dd^c u_t)^{p} \wedge T_t = \int_{\Omega_j} (dd^c v_t)^{p} \wedge T_t.$$
 On the other hand, we have
 $$
 \liminf_{t \rightarrow + \infty} \int_{\Omega_j} (dd^c v_t)^{p} \wedge T_t \leqslant \limsup_{t \rightarrow + \infty} \int_{\overline{\Omega_j}} (dd^c u_t)^{p} \wedge T_t.
 $$
Therefore, 
	$$ \int_{\Omega_j} (dd^c v)^{p} \wedge T \leqslant  \int_{\overline{\Omega_j}} (dd^c u)^{p} \wedge T.$$
	Taking the limit on $j$, we obtain
 \begin{equation}\label{E1}
 \int_{\Omega''} (dd^c u)^{p} \wedge T \leqslant \int_{\Omega''} (dd^c v)^{p} \wedge T.
 \end{equation}
 Permuting $u$ and $v$, we obtain the  equality.
\end{proof}
\subsubsection{Capacity estimates}
Given a class $\mathcal K \subset \mathcal E(\Omega)$, we denote by $\mathcal K^a$ the set of functions $u \in \mathcal K$ such that $(dd^c u)^n$ puts no mass on pluripolar sets. 

\begin{prop}\label{prop: comparison of capacities}
 Let $\psi\in \mathcal F^a(\Omega)$. Then there exists a continuous function $f : \R^+ \longrightarrow  \R^+$ with $f(0)=0$ such that, for all Borel sets $E$, 
 \[
 Cap_{\psi}(E) \leq f(Cap_{BT}(E)).
 \]
\end{prop}
The generalized capacity $Cap_{\psi}$ is defined as 
\[
Cap_{\psi}(E) = \sup \left \{\int_E (dd^c u)^n \; : \; u\in PSH(\Omega), \; \psi\leq u\leq 0 \right \}. 
\]
\begin{proof}
	Since $\psi \in \mathcal F^a(\Omega)$, by definition the measure $\mu_\psi := (dd^c \psi)^n$ has finite mass and puts no mass on the set  $\{ \psi = - \infty\}$. By a classical result there exists a  concave increasing function $h : \mathbb R^+ \rightarrow \mathbb R^+$ such that $h(+\infty)=+\infty$ and $\int_{\Omega}h(-2\psi) (dd^c \psi)^n <+\infty$. 
 
 Take a psh function $u$ such that $\psi \leq u \leq 0$ and take a Borel set $E$. For $t>0$ we have  
 \begin{flalign*}
 \int_E (dd^c u)^n & = \int_{E\cap \{\psi\leq -t\}} (dd^c u)^n + \int_{E\cap \{\psi>-t\}} (dd^c \max(u,-t))^n\\
 &\leq \frac{1}{h(t)}\int_{\Omega} h(-u) (dd^c u)^n + t^n \int_{E} (dd^c \max(u/t,-1))^n\\
 &\leq \frac{1}{h(t)} 2^n\int_{\Omega} h(-2\psi) (dd^c \psi)^n +  t^n Cap_{BT}(E)\\
 & \leq \frac{C}{h(t)} +  t^n Cap_{BT}(E).
 \end{flalign*}
	We next choose $t>0$ such that $t^{2n} Cap_{BT}(E)= 1$. The result holds with 
	\[
	f(t)=\frac{C}{h(t^{-1/2n})} + \sqrt{t}.
	\]
\end{proof}
\subsubsection{The convergence theorem}
\begin{prop}\label{prop: convergence MA quasi-cont} 
	Let $(u_{1,j})_{j \in \N}, \cdots, (u_{n,j})_{j \in \N}$  be sequences in $\mathcal E^a(\Omega)$ decreasing to $u_1, \cdots, u_n \in \mathcal E^a(\Omega)$ respectively and let $(f_j)_{j \in \N}$ be a sequence of uniformly bounded quasi-continuous functions converging in capacity to a bounded quasi-continuous function $f$. Then 
 \[
 f_j dd^c u_{1,j} \wedge ... dd^c u_{n,j} \to f dd^c u_1 \wedge...dd^c u_n
 \] 
 in the weak sense of Radon measures on $\Omega$. 
\end{prop}
\begin{proof}
We fix $D\Subset \Omega$ and prove the weak convergence in $D$. Fixing $\varepsilon>0$, by quasi-continuity, there exists an open set $G\subset D$ and a continuous function $g \in C^0(D)$ such that $Cap_{BT}(G) \leq \varepsilon$ and $f=g$ on $D\setminus G$.  By continuity of the mixed Monge-Amp\`ere operator, we have 
	\[
	g dd^c u_{1,j} \wedge ... dd^c u_{n,j} \to g dd^c u_1 \wedge...dd^c u_n
	\]
	weakly in $D$. To prove the result for the case when $f_j=f$, it remains to prove that 
	\[
	\int_D |g-f| \left (dd^c u_{1,j} \wedge ... dd^c u_{n,j} + dd^c u_{1} \wedge ... dd^c u_{n}\right)  \leq \delta(\varepsilon),
	\]
	where $\delta(\varepsilon)$ is independent of $j$ and it converges to $0$ as $\varepsilon\to 0$. Since $f,g$ are bounded and $f=g$ on $D\setminus G$ the above quantity is smaller than 
	\[
	C \int_G (dd^c v_j)^n,
	\]
	where $2n v_j = \sum_{p=1}^n (u_{p,j}+u_p)$. Setting $\psi:=(u_1+...v_n)/n$, we obtain (since $u_{p,j}\geq u_p$)
	\[
	\int_G (dd^c v_j)^n \leq n^n Cap_{\psi}(G).
	\]
	Using Proposition \ref{prop: comparison of capacities}, we finish the proof for the first case. 
	
	To treat the general case, we recall that
	\[
	Cap_{BT}(|f_j-f|\geq \varepsilon) \to 0.
	\]
	Thus, using  Proposition \ref{prop: comparison of capacities} again we have 
	\[
	\int_D |f_j-f| dd^c u_{1,j} \wedge ... dd^c u_{n,j}  \leq \eta(\varepsilon)  + \varepsilon \int_D dd^c u_{1,j} \wedge...\wedge dd^c u_{n,j}.
	\]
	with $\eta(\varepsilon) \to 0$ as $\varepsilon\to 0$. It remains to prove that the last integral is bounded in $j$. To see this we define $\tilde{u}_{p,j}:= P({\bf 1}_{D}u_{p,j})\in \mathcal F^a$. Then 
	\[
	\int_D dd^c u_{1,j} \wedge...\wedge dd^c u_{n,j} \leq \int_{\Omega} dd^c \tilde{u}_{1,j} \wedge...\wedge dd^c \tilde{u}_{n,j} \leq \int_{\Omega} dd^c \tilde{u}_{1} \wedge...\wedge dd^c \tilde{u}_{n}<+\infty.
	\] 
\end{proof}
\subsubsection{The comparison principle}

\begin{theorem}\label{comparaison principe in N} 
 Let  
 $u, v \in \mathcal{N}_\phi^a(\Omega)$ and $w_1, \cdots w_{n-p} \in \mathcal{E}^a(\Omega)$ ($1 \leq p \leq n-1$). Then for any $s \geq 0$, we have
 $$
 \int_{\{u<v-s\}} (dd^cv)^{p} \wedge T   
 \leqslant \int_{\{u< v-s\}} (dd^c u)^{p} \wedge T,
 $$
 where  $T := dd^c w_1 \wedge \cdots \wedge dd^c w_{n-p}$ if $1 \leq p \leq n-1$ and $T = {\bf 1}_\Omega$ if $p = n$.
\end{theorem}
When $T={\bf 1}_{\Omega}$, our result is a particular case of \cite[Theorem 3.2]{ACCP09}. We provide below a slightly different proof using the same idea. 
\begin{proof}
	Let $\Omega_k\nearrow \Omega$ be a fundamental sequence of strictly pseudoconvex domains  in $\Omega$. There is $\psi \in \mathcal{N}_0^a(\Omega)$ such that $u \geqslant \phi + \psi$. Let us define the following function 
 $$
 h_k = P(\mathbf{1}_{\Omega \setminus \Omega_k} (u-\phi)) \geqslant  P(\mathbf{1}_{\Omega \setminus \Omega_k} \psi).
 $$
 We claim that $h_k\nearrow 0$ q.e. in $\Omega$. Indeed it is clear that $h_k\nearrow h$ q.e. in $\Omega$, where $h \in {\rm PSH}(\Omega)$ and $\psi \leq h \leq 0$.  Hence $h \in \mathcal N^a_0(\Omega)$. Since for a fixed $\ell \in \N,$ we have $(dd^c h_k)^n = 0$ in $\Omega_\ell$ for any $k \geq \ell$, it follows that  $(dd^c h)^n = 0$  in $\Omega$.  Hence $h = 0$ in $\Omega$ by uniqueness in $\mathcal N^a_0(\Omega)$ (see \cite[Theorem 4.1]{Ceg08}).
 
 Fix $s \geq 0$ and set $v_{k} = \max(v+h_k-s, u)$. Since $h_k + v \leq h_k + \phi \leq u$ on $\Omega \setminus \Omega_{k+1}$, we have $v_{k} = u$  sur $\Omega \setminus \Omega_{k+1}$. 
 Then, by Lemma \ref{lem: conservation mass}, for any fixed $\ell > k+1$,  we have
	$$\int_{\Omega_\ell} (dd^c v_{k})^{n-p} \wedge T = \int_{\Omega_\ell} (dd^c u)^{n-p} \wedge T.$$
 Hence by splitting each integral in two parts we obtain
 \begin{flalign*}
& \int_{\Omega_\ell \cap \{u \leq v+h_k-s\}} (dd^c v_{k})^{p} \wedge T + \int_{\Omega_\ell \cap\{u > v+h_k-s\}} (dd^c v_{k})^{p} \wedge T \\
& = \int_{\Omega_\ell \cap \{u \leq v+h_k-s\}} (dd^c u)^{p} \wedge T + \int_{\Omega_\ell \cap\{u > v+h_k-s\}} (dd^c u)^{p} \wedge T.
 \end{flalign*}
	
 By the pluripotential maximum principle, we have
 $$
 \int_{\Omega_\ell \cap\{u > v+h_k-s\}} (dd^c v_{k})^{p} \wedge T = \int_{\Omega_\ell \cap\{u > v+h_k-s\}} (dd^c u)^{p} \wedge T.
 $$
 Therefore
 $$
 \int_{\Omega_\ell \cap \{u \leq v+h_k-s\}} (dd^c v_{k})^{p} \wedge T = \int_{\Omega_\ell \cap \{u \leq v+h_k-s\}} (dd^c u)^{p} \wedge T.
 $$
  Applying again the pluripotential maximum principle we obtain
 $$
  \int_{\Omega_\ell \cap \{u < v+h_k - s- \varepsilon \}} (dd^c v_{k})^{p} \wedge T =
  \int_{\Omega_\ell \cap \{u < v+h_k - s- \varepsilon\}} (dd^c (v + h_k))^{p} \wedge T.
 $$
 Now observe that  $(dd^c v)^p \wedge T \leq (dd^c (v+h_k))^p \wedge T$. Then   we have

\begin{flalign*}
 \int_{\Omega_\ell \cap \{u < v+h_k - s-\varepsilon\}} (dd^c v)^{p} \wedge T & \leq \int_{\Omega_\ell \cap \{u \leq v+h_k - s\}} (dd^c u)^{p} \wedge T \\
 &  \leq \int_{\Omega_\ell \cap \{u < v - s  \}} (dd^c u)^{p} \wedge T.
\end{flalign*}
Since $h_k \nearrow 0$ as $k \to + \infty$, letting first $k \to +\infty$, $\varepsilon \to 0$ and $\ell \to + \infty$ we obtain the required inequality.
\end{proof}

\subsection{The projection theorem} 
Envelopes of plurisubharmonic functions are classical objects (see \cite{GZbook}). Here we follow the presentation in \cite{GLZ19}.

Let $f : \Omega \longrightarrow \R \cup \{-\infty\}$ be quasi-continuous function defined quasi-everywhere in $\Omega$, which is locally bounded from above. A typical example is when $f =\min(w_1 - w_2,0)$ where $w_1, w_2$ are plurisubharmonic functions in $\Omega$. Then $f$ is well defined on $\Omega \setminus \{w_2 = - \infty\}$ and is quasi-continuous in $\Omega$ by \cite{BT82}.

 We define the plurisubharmonic envelope of $f$ as follows:
$$
P(f) = P_\Omega (f) := \sup \{v \in {\rm PSH}(\Omega) ; \, v \leq f \, \, \text{q.e. in}\, \,  \Omega\}.
$$
Here $v \leq f$ q.e. in $\Omega$ means that the inequality is satisfied outside a pluripolar set. 

\begin{lem} \label{lem:MA contact set} Assume there exists $\psi \in \mathcal E^a(\Omega)$ such that $\psi \leq f$ q.e. in $\Omega$.
Then $P(f) \in \mathcal E^a(\Omega)$, $P(f) \leq f$ q.e. in $\Omega$ and satisfies the following property
\begin{equation} \label{eq:orthogonality}
{\bf 1}_{\{P(f) < f\}}  (dd^c P(f))^n = 0, 
\end{equation}
in the sense of Borel measures on $\Omega$ i.e. the Monge-Ampère measure $(dd^c P(f))^n$ is carried on the contact set $\{P(f) = f\}$.
\end{lem}
Observe that when $f \in PSH(\Omega)$ then $P(f) = f$; hence $P^2(f) = P(f)$ for any function $f$. Moreover 
the equation \eqref{eq:orthogonality} is  a kind of orthogonality property. Indeed it implies that for any continuous  test function $\chi$ in $\Omega$, we have 
$$
\int_\Omega (f - P(f)) \chi (dd^c P(f))^n = 0.
$$
In this sense $P$ acts as an "orthogonal projection" onto the convex set $PSH(\Omega)$. 
\begin{proof}
Since $\psi \in \mathcal V(f) := \{v \in {\rm PSH}(\Omega) ; \, v \leq f \, \, \text{q.e. in}\, \,  \Omega\}$, the upper semi-continuous regularization  of $P(f)$ is plurisubharmonic i.e. $v := P(f)^* \in PSH(\Omega)$ and $\psi \leq P(f)$ q.e. in $\Omega$. Since $\{P(f) < P(f)^*\}$ is pluripolar by \cite{BT82}, it follows that $ v \leq P(f)$ q.e. in $\Omega$.
Moreover by Choquet's Lemma $v = (\sup_{j\in \N} v_j)^*$ where $v_j \in \mathcal V(f)$. Hence 
$v \leq f$ q.e. in $\Omega$ since the union of a countable pluripolar sets is pluripolar. Therefore  $\psi \leq v = P(f)$ which proves that $P(f)^* = P(f) \in \mathcal E^a (\Omega)$ in $\Omega$.

We now prove the formula \eqref{eq:orthogonality} by following the arguments of \cite[Theorem 2.7]{ACLR25} and \cite[Theorem 2.7]{GLZ19}. Let $(\Omega_k)$ be a sequence of relatively compact open subsets of $\Omega$ that increase to $\Omega$ and let $u_k$ be the largest psh function in $\Omega_k$ lying below $f$ in $\Omega$. Since $f$ is bounded from above in $\Omega_k$, by \cite[Theorem 2.7]{ACLR25} and \cite[Theorem 2.7]{GLZ19}, we have 
\[
\int_{\{u_k<f\}} (dd^c u_k)^n =0. 
\]
Observe also that $u_k \searrow P(f)$ as $k\to +\infty$. We can therefore argue as in \cite[Theorem 2.7]{ACLR25} to get \eqref{eq:orthogonality} by letting $k\to +\infty$.  
\end{proof}

From this result we deduce the following pluripotential minimum principle which was proved in \cite[Lemma 4.1]{GLZ19} (see also \cite[Corollary 2.9]{ACLR25} for an elementary proof).  
\begin{proposition} \label{prop:principe min} Let $u,v \in \mathcal E^a(\Omega)$. Assume that there exists $\psi \in \mathcal E^a(\Omega)$ such that $\psi \leq \inf \{u,v\}$ in $\Omega$. Then $w:= P(u,v) := P(\min\{u,v\}) \in \mathcal E^a(\Omega)$ and satisfies the following inequality  
 $$
 (dd^c w)^n \leq {\bf 1}_{\{ u\leq v\}} (dd^c u)^n + {\bf 1}_{\{u > v\}} (dd^c v)^n.
 $$
\end{proposition}
%\begin{proof}
%Since $\psi \leq \min\{u,v\}$, the function $w:= P(u,v) := P(\min\{u,v\})$ is well defined and satisfies $\psi \leq w \leq \min\{u,v\}$. By  Lemma \ref{lem:MA contact set} $(dd^c w)^n$ is carried by the contact set 
% $$
% \{w = \min(u,v)\} = \{w = u\} \cup \{w = v\}   = \{u \leq v\} \cup \{u > v\}\cdot
% $$
% Therefore 
%$$
%(dd^c w)^n \leq {\bf 1}_{\{u \leq v\}} (dd^c w)^n + {\bf 1}_{ \{u >  v\}} (dd^c w)^n.
%$$
%On the other hand, since $w \leq u$  and $ \{u \leq v\} \subset \{w = u\}$, it follows from the maximum principle Theorem \ref{thm:principe du max} that ${\bf 1}_{\{u \leq v\}} (dd^c w)^n \leq {\bf 1}_{\{u \leq v\}} (dd^c u)^n$. In the same way we deduce that ${\bf 1}_{\{u > v\}} (dd^c w)^n \leq {\bf 1}_{\{u > v\}} (dd^c v)^n$.
%Hence 
%$$
%(dd^c w)^n \leq {\bf 1}_{\{u \leq v\}} (dd^c u)^n + {\bf 1}_{ \{u >  v\}} (dd^c v)^n.
%$$
%\end{proof}

\section{A new  class of psh functions with finite energy}\label{sect: new energy class}
\subsection{The class $\mathcal E^1_{\phi}(\Omega)$}
Let $\phi \in \mathcal M(\Omega).$ We will introduce a new class of plurisubharmonic functions with boundary value $\phi$ and finite energy.

$$
\mathcal{E}_\phi^1(\Omega)=\{u \in \mathcal E (\Omega)\, ; \;  \exists \, v \in \mathcal E^1(\Omega) \, ; \, v + \phi \leq u \leq \phi, \, ~ \, e_\phi(u) <+\infty \},
$$
where
$$
e_\phi(u)  := \int_\Omega (\phi-u) (dd^c u)^n,
$$
is  the partial $\phi$-energy for $u \in \mathcal E_\phi^1(\Omega)$.
%Observe that it follows from the definition that the measure $(dd^c u)^n$ puts no mass on the set $\{ u = - \infty\} \setminus \{\phi = - \infty\}$. 

It follows also  from the definition that if  $v + \phi \leq u \leq \phi$  and $v \in \mathcal E^1(\Omega)$ then $ v \leq P(u-\phi) \leq 0$, hence $P(u-\phi) \in \mathcal{E}^1(\Omega)$. In some sense this is a way to express that $u-\phi$ has "boundary values" equal to $0$. Observe that when $n=1$, $\phi$ is harmonic and then $P(u-\phi) = u - \phi$ and  $e_\phi(u) = \int_{\Omega}(\phi-u) dd^c (u - \phi) = \Vert \nabla (u-\phi)\Vert^2$. This means that $u-\phi \in W^{1,2}_0(\Omega)$. 
%Since $\mathcal{E}^1(\Omega) \subset L^p(\Omega)$ for any $p > 1$, it follows that $\phi - u \in L^p(\Omega)$ and then u \in \subset L^p(\Omega)$ for any $p > 1$.

\begin{lemma} \label{lem:increasing}
 Let $u \in \mathcal{E}_\phi^1(\Omega)$. Then for any $1\leq p \leq n-1$, we have
 \begin{flalign*}
 \int_{\Omega} (\phi-u) (dd^c u)^p \wedge (dd^c \phi)^{n-p}  
 &\leq  \int_{\Omega} (\phi-u) (dd^c u)^{p+1}  \wedge (dd^c \phi)^{n-p-1} \\
 &\leq \int_{\Omega} (\phi-u) (dd^c u)^{n} < \infty.
\end{flalign*}
In particular  all the mixed terms are finite.
\end{lemma}
 \begin{proof} By Fubini's theorem
 \begin{flalign*}
 \int_{\Omega} (\phi-u) (dd^c u)^p \wedge (dd^c \phi)^{n-p}   
 = \int_{0}^{+\infty}  d t \int_{\{ u < \phi-t \}} dd^c \phi \wedge (dd^c u)^p \wedge (dd^c \phi)^{n-p-1} .
  \end{flalign*}
 By the comparaison principe Theorem \ref{comparaison principe in N}, it follows that
 \begin{flalign*}
 \int_{\Omega} (\phi-u) (dd^c u)^p \wedge (dd^c \phi)^{n-p}&  \leqslant   \int_{0}^{+\infty} d t \int_{\{u <\phi-t\}}  (dd^c u)^{p+1} \wedge (dd^c \phi)^{n-p-1}\\
 &= \int_\Omega (\phi-u) (dd^c u)^{p+1} \wedge (dd^c \phi)^{n-p-1}.
	\end{flalign*}
\end{proof}

\begin{lemma}\label{comparaison principe boundary value}
Let $v \in \mathcal{N}^a_\phi(\Omega)$ and $u \in PSH(\Omega)$  such that $v \leqslant u \leq \phi$ and 
$w_1, \cdots, w_{n-p} \in \mathcal E^a (\Omega)$ with $1 \leq p \leq n-1$. Let  $h: \mathbb R^+ \rightarrow \mathbb R^+$ be an increasing function.  Then  $u \in  \mathcal{N}^a_\phi(\Omega)$ and we have
$$
\int_{\Omega} h(\phi-u)(dd^c u)^p \wedge T \leqslant 2^p \int_{\Omega} h(2\phi-2v)(dd^cv)^p\wedge T,
$$
where $T = dd^c w_1 \wedge \cdots  dd^c w_{n-p}$ if $1 \leq p\leq n-1$ and $T = {\bf 1}_{\Omega}$ if $p=n$.

In particular if $v \in  \mathcal{E}_\phi^1(\Omega)$ then $u \in  \mathcal{E}_\phi^1(\Omega)$ and $e_\phi (u) \leq 2^{n+1} e_\phi(v).$
\end{lemma}

\begin{proof}
It is clear from the definition that $u \in \mathcal N_\phi(\Omega)$. We need to prove  that  $u \in \mathcal E^a(\Omega)$. We can assume that $v \in \mathcal F^a (\Omega), u \in \mathcal F (\Omega)$ and $v \leq u$.  Then by the characterization of the class $\mathcal F^a (\Omega)$ in terms of  finite  weighted energy, it follows that  $u  \in \mathcal F^a (\Omega)$ (see \cite[Proposition 4.2, Corollary 4.4]{BGZ09}.

We have
$$
I:= \int_{\Omega}	h(\phi-u)(dd^cu)^p\wedge T  = \int_{0}^{+\infty} d h(s)  \int_{\{u < \phi- s \}} (dd^cu)^p\wedge T.
$$
 Since $v \leq u$, we have $\{u < \phi- s\} \subset \{v < \frac{u+\phi}{2}- \frac{s}{2}\}$ for any $s > 0$ and then
\begin{flalign*}
 \int_{\{u < \phi- s\}} (dd^cu)^p\wedge T &\leqslant  \int_{\{v < \frac{u+\phi}{2} - \frac{s}{2}\}} (dd^cu)^p \wedge T\\
 &  \leqslant 2^p \int_{\{v < \frac{u+\phi}{2} - \frac{s}{2}\}} \left(dd^c (\frac{u+\phi}{2})\right)^p\wedge T.
\end{flalign*}
Observe that $u, \phi \in \mathcal N_{\phi}^a(\Omega)$ and  then $\frac{u+\phi}{2} \in \mathcal{N}^a_{\phi}(\Omega)$. By the comparison principle Theorem \ref{comparaison principe in N}, we have
\begin{flalign*}
	\int_{\{u < \phi- s\}} (dd^cu)^p\wedge (dd^c \phi)^{n-p} &\leqslant 2^p \int_{\{v < \frac{u+\phi}{2} - \frac{s}{2}\}} (dd^c v)^p\wedge T\\
	&\leqslant 2^p \int_{\{v-\phi < \frac{u-\phi}{2} - \frac{s}{2}\}} (dd^c v)^p\wedge T\\
	&\leqslant  2^p \int_{\{v-\phi < - \frac{s}{2}\}} (dd^c v)^p\wedge T.
\end{flalign*}
Hence 
\begin{flalign*}
	I &\leqslant  2^p\int_{0}^{+\infty} d h(s) \int_{\{v-\phi < - \frac{s}{2}\}} (dd^c v)^p\wedge T \\
	&\leqslant 2^p \int_{\Omega} h(2\phi-2v)(dd^cv)^p\wedge T.
\end{flalign*}
The last claim follows from the first one when $h(s) = s$ and $T = {\bf 1}_\Omega$.
\end{proof}

\begin{lemma}\label{lem: mixed energy inequality}
Let $u, v \in \mathcal{N}^a_\phi(\Omega)$ and 
$w_1, \cdots, w_{n-p} \in \mathcal E^a (\Omega)$ with $1 \leq p \leq n-1$. If $h: \mathbb R^+ \rightarrow \mathbb R^+$ be an increasing function then  we have
$$
\int_{\Omega} h(\phi-u)(dd^c v)^p \wedge T \leq \int_{\Omega} h(2\phi-2u)(dd^c u)^p\wedge T  +  \int_{\Omega} h(2\phi-2v)(dd^cv)^p\wedge T,
$$
where $T = dd^c w_1 \wedge \cdots  dd^c w_{n-p}$ if $1 \leq p\leq n-1$ and $T = {\bf 1}_{\Omega}$ if $p=n$.

In particular if $u, v \in  \mathcal{E}_\phi^1(\Omega)$ then
$$
\int_{\Omega} (\phi-u)(dd^c v)^n \leq 2 e_\phi (u) + 2 e_\phi (v).
$$
\end{lemma}

\begin{proof}
We can assume that the term on the right hand side is finite.
We have
$$
I:= \int_{\Omega}	h(\phi-u)(dd^c v)^p\wedge T  = \int_{0}^{+\infty} d h(s)  \int_{\{u < \phi- s \}} (dd^c v)^p\wedge T.
$$
 Since   $\{u < \phi- s\} \subset \{u  < v - \frac{s}{2}\} \cup \{v < \phi- \frac{s}{2}\}$ for any $s > 0$, we have
\begin{flalign*}
 \int_{\{u < \phi- s\}} (dd^c v)^p\wedge T &\leqslant  \int_{\{u  < v - \frac{s}{2}\}} (dd^c v)^p \wedge T +
 \int_{ \{v < \phi- \frac{s}{2}\} } (dd^c v)^p \wedge T.
\end{flalign*}
 By the comparison principle Theorem \ref{comparaison principe in N}, we have
\begin{flalign*}
	\int_{\{u < \phi- s\}} (dd^c v)^p\wedge T &\leq \int_{\{u  < v - \frac{s}{2}\}} (dd^c u)^p\wedge T + \int_{ \{v < \phi- \frac{s}{2}\} } (dd^c v)^p \wedge T.
\end{flalign*}
Since $v \leq \phi$, integrating this inequality with respect to $dh(s)$, we obtain
$$
\int_\Omega h(\phi-u) (dd^c v)^p\wedge T \leq \int_\Omega h(2\phi- 2 u) (dd^c u)^p\wedge T  + \int_\Omega h(2 \phi-2 v) (dd^c v)^p\wedge T.
$$
\end{proof}

\begin{lemma}\label{lem: mixed energy}
	Assume that $u,v \in \mathcal{E}_\phi^1(\Omega)$.  Then
 $$
 \int_{\Omega} (\phi-u)(dd^c v)^n \leqslant C(n,v) \max \left\{ e_\phi(u)^{\frac{1}{n+1}},  e_\phi(u)^{\frac{1}{2}} \right\},
 $$
 where $ e_\phi(u) := \int_{\Omega} (\phi - u)(dd^c u)^n$ and $C(n,v) := 2e_\phi(v) + n2^{n+1}$.
\end{lemma}
\begin{proof} We consider two cases.

\noindent {\bf Case 1 :}	Assume that $e_\phi (u) \leqslant 1$ and take $m\geq 1$ such that $m^{n+1}\int_{\Omega} (\phi - u)(dd^cu)^n =1$. 
 Set 
 $$ 
 w = P(mu-(m-1)\phi).
 $$ 
 Then by the pluripotential comparison principle we have $(dd^c w)^n \leqslant m^n(dd^cu)^n$ on the contact set $\{w = m u - (m-1) \phi\}$.
It follows that  
 \begin{flalign*}
 \int_{\Omega} (\phi -w)(dd^cw)^n &= \int_{\Omega} (\phi -(mu-(m-1)\phi))(dd^cw)^n\\
 &\leqslant m^{n+1} \int_{\Omega} (\phi -u)(dd^cu)^n =1.
 \end{flalign*}
 By Lemma \ref{lem: mixed energy inequality}, we have
 \begin{flalign*}
 \int_{\Omega} (\phi-w)(dd^c v)^n & \leqslant 2 \int_{\Omega} (\phi -w)(dd^cw)^n + 2\int_{\Omega} (\phi -v)(dd^cv)^n   \leqslant 2(1 + e_\phi(v)).
 \end{flalign*}	
 Since $\phi - w \geq m(\phi - u)$ we deduce that
  $$
  m \int_{\Omega} (\phi -u)(dd^cv)^n \leqslant 2(1 + e_\phi(v)).
  $$
 Hence
 $$
  \int_{\Omega} (\phi -u)(dd^cv)^n \leqslant \frac{2(1 + e_\phi(v))}{m} = \frac{C_1(n,v)}{m} ,
 $$
 where $C_1(n,v) :=  2(1 + e_\phi(v))$.
	Thus
 \begin{equation}\label{eq:petite energy}
 \int_{\Omega} (\phi-u)(dd^c v)^n \leqslant C_1(n,v) \left( \int_{\Omega} (\phi-u)(dd^c u)^n \right)^{\frac{1}{n+1}}.
 \end{equation}
	
\noindent {\bf Case 2 : } We assume that $e_\phi (u) \geqslant 1$ and set $m := e_\phi (u)^{-\frac{1}{2}} \leqslant 1$. Then we have
 \begin{flalign*}
 e_\phi(mu + (1-m)\phi) &= m\int_{\Omega} (\phi-u)(dd^c(mu + (1-m)\phi))^n\\
 &= m\int_{\Omega} (\phi-u) \sum_{J=1}^{n} m^j (1-m)^{n-j}\binom{n}{j}(dd^cu)^j \wedge (dd^c\phi)^{n-j}\\
 &\leqslant 2^n m^2\int_{\Omega} (\phi-u) \sum_{J=1}^{n}(dd^cu)^j \wedge (dd^c\phi)^{n-j}\\
 &\leqslant n 2^n m^2\int_{\Omega} (\phi-u) (dd^cu)^n\\
 &\leqslant n 2^n.
 \end{flalign*}
 Then by Lemma \ref{lem: mixed energy inequality}, we have 
 $$
 m \int_{\Omega} (\phi -u)(dd^cv)^n \leqslant 2(n2^n  + e_\phi(v)).
 $$
 So 
 $$ 
 \int_{\Omega} (\phi -u)(dd^cv)^n \leqslant  \frac{C_2(n,v)}{m},
 $$
 where $C_2(n,v) := 2^{n+1}n + 2e_\phi(v)$.
 Thus 
 \begin{equation}\label{eq:grande energy}
 \int_{\Omega} (\phi-u)(dd^c v)^n \leqslant C_2(n,v) \left( \int_{\Omega} (\phi-u)(dd^c u)^n \right)^{\frac{1}{2}}.
 \end{equation}
	From  \eqref{eq:petite energy} and \eqref{eq:grande energy}, the result follows.
\end{proof}

\begin{lemma}\label{lem: convexity}
	If $u,v\in \mathcal E^1_{\phi}(\Omega)$ then $w := P(u,v)\in \mathcal E^1_{\phi}(\Omega)$. In particular $\mathcal E^1_{\phi}(\Omega)$ is a convex set. 
\end{lemma}
\begin{proof}
We will apply  the minimum principle Proposition \ref{prop:principe min}.  By definition there exists $u_0, v_0 \in \mathcal N_0 (\Omega)^a$ such that $u_0 + \phi \leq u \leq \phi$ and  $v_0 + \phi \leq v \leq \phi$ in $\Omega$. 
 Then we have $w_0:= u_0 + v_0 \in \mathcal N_0 (\Omega)^a$ and $w_0  \leq \inf \{u_0,v_0\}$.
 Hence $\psi := w_0 + \phi \in \mathcal N^a_{\phi}(\Omega)$ and satisfies $\psi  \leq \inf \{u,v\}$.
 By Proposition \ref{prop:principe min}, it follows that  $w := P(u,v)\in \mathcal N^a_{\phi}(\Omega)$.  It remains to show that $\int_{\Omega} (\phi-w) (dd^c w)^n < + \infty$.
 
 Indeed by Proposition \ref{prop:principe min} we have 
 $$
 (dd^c w)^n \leq {\bf 1}_{\{ u \leq v \}} (dd^c u)^n + {\bf 1}_{\{u > v\}} (dd^c v)^n.
 $$
 Then  since $(dd^c w)^n$ is carried on the contact set $\{w = \inf\{u,v\}\}$ we have
	\begin{flalign*}
		\int_{\Omega} (\phi-w) (dd^c w)^n & \leq \int_{\Omega} (\phi-u) (dd^c u)^n + \int_{\Omega} (\phi-v) (dd^c v)^n<+ \infty. 
	\end{flalign*}
	As a consequence $w\in \mathcal E^1_{\phi}(\Omega)$. For any $t\in [0,1]$, we have $tu+(1-t)v \geq P(u,v)$, hence $tu+(1-t)v$ also belongs to $\mathcal E^1_{\phi}(\Omega)$ by Lemma \ref{comparaison principe boundary value}. 
\end{proof}

\subsection{Integration by parts}

\begin{theorem}\label{thm: integration by parts}
	Let $u_1,u_2,v_1,v_2, w_1, \cdots, w_{n-1} \in \mathcal E^1_{\phi}(\Omega)$ and $T=dd^c w_1 \wedge ... \wedge dd^c w_{n-1}$. Then 
	\begin{equation}
		\label{eq: integration by parts}
		\int_{\Omega}(u_1-u_2)dd^c(v_1-v_2) \wedge T = \int_{\Omega}(v_1-v_2)dd^c(u_1-u_2) \wedge T.
	\end{equation}
\end{theorem}

To prove the integration by parts in this general form we need several intermediate results. 

\begin{lemma}\label{lem: integration by parts equality near boundary}
	With the assumptions of Theorem \ref{thm: integration by parts} we assume also that $u_1=u_2$ near the boundary of $\Omega$ and all functions are relatively bounded with respect to $\phi$: $u_j,v_k,w_l \geq \phi-C$. Then \eqref{eq: integration by parts} holds. 
\end{lemma}
\begin{proof}
	Let $\Omega_1\Subset \Omega_2\Subset \Omega$ be open sets such that $u_1=u_2$ in $\Omega\setminus \Omega_1$. We approximate the functions $u_j,v_k,w_l$ by taking convolutions with smooth mollifiers $(\chi_j)$. When $j$ is sufficiently large, for each $p = 1; 2$,  the approximants  $u_{p,j}=u_p\star \chi_{j}$ are defined in $\Omega_2$, and the functions $u_{1,j}-u_{2,j}$ are compactly supported in $\Omega_2$. By Stokes theorem we have 
	\[
	\int_{\Omega_2}(u_{1,j}-u_{2,j}) dd^c (v_{1,j}-v_{2,j}) \wedge T_j =\int_{\Omega_2}(v_{1,j}-v_{2,j}) dd^c (u_{1,j}-u_{2,j}) \wedge T_j.
	\]
	Observe that $u_{1,j}-u_{2,j}$ and $v_{1,j}-v_{2,j}$ are uniformly bounded, quasi-continuous and they converge in capacity to $u_1-u_2$ and $v_1-v_2$ respectively.  We thus obtain by letting $j\to +\infty$ and using Proposition \ref{prop: convergence MA quasi-cont},
	\[
	\int_{\Omega_2}(u_{1}-u_{2}) dd^c (v_{1}-v_{2}) \wedge T =\int_{\Omega_2}(v_{1}-v_{2}) dd^c (u_{1}-u_{2}) \wedge T.
	\]
\end{proof}

The next result we need is the following. A more general version of this will be proved later. 

\begin{lemma}\label{lem: energy convergence increasing}
	Let $u\in \mathcal E^1_{\phi}(\Omega)$, $w_1, \cdots, w_{n-p} \in \mathcal E^1_{\phi} (\Omega)$ and $T=dd^c w_1 \wedge ... \wedge dd^c w_{n-p}$ when $1 \leq p \leq n -1$ and $T = {\bf 1}_\Omega$ when $p = n$. Assume $(u_j)$ is a sequence in $\mathcal E^1_{\phi}(\Omega)$ increasing to $\phi$. Assume also that  for each $j$, $u_j=u$ near $\partial \Omega$, and all the functions $u_j-\phi$, $w_j-\phi$ are bounded.  Then 
	\begin{equation}
		\label{eq: energy convergence increasing}
	\lim_{j\to +\infty}	\int_{\Omega}(\phi-u) (dd^c u_j)^p \wedge T = \int_{\Omega}(\phi-u) (dd^c \phi)^p \wedge T.
	\end{equation}
	In particular, for $p=n$, we have 
	\begin{equation}
		\label{eq: dominant energy convergence increasing}
	\lim_{j\to +\infty}	\int_{\Omega}(\phi-u_j) (dd^c u_j)^n = 0. 
	\end{equation}
\end{lemma}
\begin{proof}
The last statement follows from the first one since $(dd^c \phi)^n=0$ and $u\leq u_j\leq \phi$. We first assume that $u-\phi,w_j-\phi$ are uniformly bounded. 

	We prove the first statement by induction on $p$. For $p=0$, there is nothing to prove. We assume the statement is true for $0\leq p-1< n$ and we prove that it is still true for $p$. Setting $S_j:= (dd^c u_j)^{p-1}\wedge T$ and using integration by parts (since the support of $u_j-u$ is compact), we have
	\begin{flalign*}
		&\int_{\Omega} (\phi-u) (dd^c u_j) \wedge S_j = \int_{\Omega} (\phi-u) (dd^c (u_j-u) + dd^c u) \wedge S_j  \\
		& = \int_{\Omega} (u_j-u) dd^c (\phi-u) \wedge S_j  +  \int_{\Omega} (\phi-u) dd^c u \wedge S_j\\
		&= \int_{\Omega} (\phi-u) dd^c (\phi-u) \wedge S_j -\int_{\Omega} (\phi-u_j) dd^c (\phi-u) \wedge S_j +  \int_{\Omega} (\phi-u) dd^c u \wedge S_j\\
			& = \int_{\Omega} (\phi-u) dd^c \phi \wedge S_j - \int_{\Omega} (\phi-u_j) dd^c (\phi-u) \wedge S_j.
	\end{flalign*}
	By the induction hypothesis, we have 
	\[
	\lim_j \int_{\Omega} (\phi-u) dd^c \phi \wedge S_j =  \int_{\Omega} (\phi-u)  \wedge (dd^c \phi)^{p} \wedge T. 
	\]
	For fixed $k\in \mathbb N$ and for $j\geq k$, we also have 
	\begin{flalign*}
	 \int_{\Omega} (\phi-u_j) dd^c \phi \wedge S_j  \leq  \int_{\Omega} (\phi-u_k) dd^c \phi \wedge S_j  
	\end{flalign*}
	Since $u_j=u_k$ near $\partial \Omega$, we can use the induction hypothesis to get 
	\begin{flalign*}
	\limsup_{j} \int_{\Omega} (\phi-u_j) dd^c \phi \wedge S_j  & \leq  \limsup_j \int_{\Omega} (\phi-u_k) dd^c \phi \wedge S_j  \\
	&= \int_{\Omega} (\phi-u_k) (dd^c \phi)^{p} \wedge T. 
	\end{flalign*}
	Letting $k\to +\infty$ and using the dominated convergence theorem, we arrive at 
	\[
	\limsup_{j} \int_{\Omega} (\phi-u_j) dd^c \phi \wedge S_j  =0. 
	\]
	The same argument also yields 
	\[
	\limsup_{j} \int_{\Omega} (\phi-u_j) dd^c u \wedge S_j  =0,
	\]
	finishing the proof in the first case. 
	
	To prove the general case we consider the canonical approximants and argue exactly as in the second step of the proof of Theorem \ref{thm: integration by parts}. 
\end{proof}

We are now ready to prove Theorem \ref{thm: integration by parts}. 
\begin{proof}[Proof of Theorem \ref{thm: integration by parts}]
We proceed in two steps. 

{\bf Step 1.}
	We  assume first that $u_l,v_l,w_j \geq \phi-C$, for some constant $C>0$ and proceed in several steps.   
	
By bilinearity, we can assume that $u=u_1$, $\phi=u_2=v_2$, $v=v_1$ and prove that
 $$
 \int_{\Omega} (\phi-u) dd^c (\phi-v) \wedge T = \int_{\Omega} (\phi-v) dd^c (\phi-u) \wedge T.
 $$
 This numbers may be infinite. We will assume that $I:=  \int_{\Omega} (\phi-u) dd^c (\phi-v) \wedge T$ is finite and prove the equality.
We let $\Omega_k$ be a sequence of relatively compact subsets of $\Omega$ that increase to $\Omega$, and define for $k \in \N$
 $$
  u_{k}:= P({\bf 1}_{\Omega\setminus \Omega_k} u), \, \, u_k:=P({\bf 1}_{\Omega\setminus \Omega_k}v).
 $$ 
 Observe that $u_k \nearrow \widetilde u = \phi$ and $v_k \nearrow \widetilde v = \phi$ as $k \to + \infty$.
	
We have  for fixed $k \in \N$, 
 \begin{flalign*}
 \int_{\Omega} (\phi-u) dd^c (\phi-v) \wedge T & = \int_{\Omega} (u_k-u+ \phi-u_k) dd^c (\phi-v)\wedge T  \\
 & = \int_{\Omega}(\phi-u_k) dd^c (\phi-v) \wedge T + \int_{\Omega} (u_k-u) dd^c (\phi-v)\wedge T.
 \end{flalign*}
 Applying the first step we have for a fixed $k\in \N$,
 \begin{equation} \label{eq:symmetry}
  \int_{\Omega} (u_k-u) dd^c (\phi-v)\wedge T = \int_{\Omega} (\phi-v) dd^c (u_k-u)\wedge T.
 \end{equation}
 Hence 
 \begin{flalign*}
 \int_{\Omega} (\phi-u) dd^c (\phi-v) \wedge T & = \int_{\Omega}(\phi-u_k) dd^c (\phi-v) \wedge T + \int_{\Omega} (\phi-v) dd^c (u_k-u)\wedge T\\
 & = \int_{\Omega} (\phi-v) dd^c (\phi-u) \wedge T + \varepsilon_k,
 \end{flalign*}
 where 
 $$\varepsilon_k := \int_{\Omega}(\phi-\phi_k) dd^c (\phi-v) \wedge T + \int_{\Omega} (\phi-v) dd^c (u_k-\phi) \wedge T.
 $$
 It remains to show that $\varepsilon_k \to 0$ as $k\to +\infty$. By the dominated convergence theorem the first term defining $\varepsilon_k$ converges to $0$. To treat the second term  we observe that for fixed $k \in \N$, the dominated convergence theorem gives  
 $$
 \int_{\Omega}  (\phi-v) dd^c (u_k-\phi) \wedge T =\lim_j  \int_{\Omega} (v_j-v) dd^c (u_k -\phi)\wedge T.
 $$
 Since $v_j=v$ near $\partial \Omega$, we can use Lemma \ref{lem: integration by parts equality near boundary} to get 
 \[
 \int_{\Omega} (v_j-v) dd^c (u_k -\phi)\wedge T = \int_{\Omega} (u_k-\phi) dd^c (v_j-v)\wedge T. 
 \]
By Lemma \ref{lem: mixed energy}, it remains to prove that 
\[
\lim_k \int_{\Omega} (\phi-u_k) (dd^c u_k)^n = 0.
\]
Since $u_k=u$ near $\partial \Omega$, the last convergence follows from Lemma \ref{lem: energy convergence increasing}, finishing the proof of the first step. 

{\bf Step 2.}   
We now treat the general case by considering $u^t:= \max(u,\phi-t)$. It suffices to prove that 
 \begin{flalign*}
 \int_{\{u\leq \phi-t\}}  (\phi-u^t) dd^c v^t \wedge T \leq \delta(t),
 \end{flalign*}
	with $\delta(t)\to 0$ as $t\to +\infty$.
 By convexity Corollary \ref{monotonicity-convexity}, $\psi:= (w_1+ \cdots +w_{n-1})\slash (n-1) \in \mathcal E^1_{\phi} (\Omega)$ and we have $T \leq S := c_n (dd^c \psi)^{n-1}$.
 
 Now take a convex increasing function $h : \mathbb R^+ \rightarrow \mathbb R^+$ such that $h(t)/t \to +\infty$ as $t\to +\infty$ and
 \[
M:= \int_{\Omega} h(2(\phi-u)) dd^c u \wedge S + \int_{\Omega} h(2(\phi-v)) dd^c v \wedge S <+\infty,
	\] 
 Observe that  
\begin{flalign*}
\int_{\{u\leq \phi-t\}}  (\phi-u^t) dd^c v^t \wedge S &= \frac{t}{h(t)} \int_\Omega h (\phi-u^t) dd^c v^t \wedge S.
\end{flalign*}
Since $\lim_{t \to +\infty} \frac{t}{h(t)} = 0$, it is enough to show that 
it is enough to prove that the function
  $t \longmapsto \int_{\Omega} h(\phi-u^t) dd^c v^t \wedge S$ is bounded in $t > 0$.
 
 Applying Lemma \ref{lem: mixed energy inequality} we obtain for ay $t >0$,
$$
\int_\Omega h (\phi-u^t) dd^c v^t \wedge S \leq \int_\Omega h (2(\phi-u^t)) dd^c u^t \wedge S
+  \int_\Omega h (2(\phi-v^t)) dd^c v^t \wedge S.
$$
Applying Lemma \ref{comparaison principe boundary value} we finally get for any $t >0$,
\begin{flalign*}
\int_\Omega h (\phi-u^t) dd^c v^t \wedge S  & \leq 2^{p}  \int_\Omega h (2(\phi-u)) dd^c u \wedge S 
 +  2^p \int_\Omega h (2(\phi-v)) dd^c v \wedge S
 \leq 2^p M,
 \end{flalign*}
 which proves our claim and finishes the proof.
 \end{proof}
 
\begin{prop} \label{prop:positivity}
 Let  $u, v, w_1, \cdots, w_{n-1} \in \mathcal E^{1}_{\phi}(\Omega)$ and $T:=dd^c w_1 \wedge ... \wedge dd^c w_{n-1}$.  Then 
 \begin{equation}
 \label{eq: integration by parts bis}
  \int_{\Omega}(u-v)dd^c(v-u) \wedge T \geq 0. 
 \end{equation}
 \end{prop}
 \begin{proof}
First observe that the integral $I(u,v,T) := \int_{\Omega}(u-v)dd^c(v-u) \wedge T$ is finite.
Indeed by bilinearity, it's enough to show that the integrals like  
$I (\phi,u,v) = \int(\phi-u) dd^c v \wedge T$ are finite. To prove this claim observe that  Lemma \ref{lem: convexity} $w := (v + w_1 + \cdots + w_{n-1})\slash n \in  \mathcal E^{1}_{\phi} (\Omega)$ and then $I (\phi,u,v) \leq c_n \int_\Omega (\phi - u) (dd^c w)^n$ which is finite by Lemma \ref{lem: mixed energy inequality}.

To prove the positivity we argue in two steps.

We first assume $u\geq v$. Then 
 \begin{flalign*}
 \int_{\Omega}(u-v)dd^c(v-u) \wedge T  &=  \int_0^{+\infty} dt \int_{\{v<u-t\}}dd^c (v-u) \wedge T \geq 0 \\
 & \leq  \int_0^{+\infty} dt \int_{\{v<u-t\}}dd^c v \wedge T - \int_0^{+\infty} dt \int_{\{v<u-t\}} 
 dd^c u \wedge T \geq 0, 
 \end{flalign*}
	by the comparison principle Proposition \ref{comparaison principe in N}. 
	
	For the general case we set $w=\max\{u,v\}$ and use the first step to get
	\begin{flalign*}
		\int_{\Omega}(u-v)dd^c(v-u) \wedge T & =\int_{\Omega}(u-w+ w-v)dd^c(v-w+w-u) \wedge T\\
		& \geq \int_{\Omega}(u-w)dd^c(v-w) \wedge T + \int_{\Omega}(w-v)dd^c(w-u) \wedge T.
	\end{flalign*}
	Each term in the last line is $0$ by the pluripotential maximum principle. 
\end{proof}

%\begin{corollary} Let $u, v \in \mathcal N(\Omega,\phi)$ such that $v \leq u$ and $\psi \in PSH(\Omega) \cap L^{\infty}$. Then for any $k \geq 1$ 
%$$
%\int_\Omega (v-u)^{n + k} (dd^c \psi)^n  \leq c_{n,k}  \Vert \psi\Vert_{\infty}^n\int_\Omega (v-u)^{k} (dd^c u)^n 
%$$
%\end{corollary}
%\begin{proof} Assume first that $u $

%\end{proof}
\begin{prop} \label{prop:uniqueness}
	Let $u, v \in \mathcal{E}_\phi^1(\Omega)$ such that $(dd^c u)^n \leq (dd^c v)^n$ on the Borel set $\{u \leq v\}$. Then $u \geqslant v$.
\end{prop}
When $\phi$ is bounded from above, the result was known (see e.g.  \cite[Corollary 3.2]{ACCP09}, \cite{ACLR25}). The general case can be proved similarly as soon as we have the Projection Theorem. For completeness we include it below.  
\begin{proof} 
  By definition $\psi:= P(u-\phi) \in \mathcal E^1(\Omega)$ and $ \psi \leq u - v$. It then follows that  $w = P(\min(u-v,0)) \in \mathcal{E}^1(\Omega)$ and $\psi \leq w \leq 0$.  Since $v+w \leq u$ in $\Omega$, from the pluripotential maximum principle we obtain 
	$$\mathbf{1}_{\{v+w = u\}} (dd^c(v+w))^n \leq \mathbf{1}_{\{v+w =  u\}} (dd^cu)^n,$$
	hence
	$$\mathbf{1}_{\{v+w = u\}} (dd^cv)^n + \mathbf{1}_{\{v+w  = u\}} (dd^cw)^n \leq  \mathbf{1}_{\{v+w = u\}} (dd^cv)^n.$$
	So $$\mathbf{1}_{\{v+w = u\}} (dd^cw)^n = 0.$$
	Since $(dd^c w)^n$ is concentrated on the contact set, it follows that  $(dd^c w)^n = 0$, hence  $w=0$ by uniqueness in $\mathcal E^1(\Omega)$ (see \cite{Ceg98}).  Therefore $u \geq v$ in $\Omega$.
 \end{proof}

\subsection{The $\phi$-energy functional}
\,\\
Let $\Omega$ be a bounded hyperconvex domain in $\mathbb C^n$, $\phi \in \mathcal M(\Omega)$. We study the Dirichlet problem 
\[
(dd^c u)^n =  \mu, \; u \in \mathcal E^1_\phi(\Omega)
\]
using the variational approach. We define the   $\phi$-energy functional
$$E_{\phi}(u)= \frac{1}{n+1} \sum_{p=0}^n \int_{\Omega} (\phi-u) (dd^c u)^p\wedge (dd^c \phi)^{n-p}.$$
It follows from  Lemma \ref{lem:increasing} that $E_{\phi}(u) < +\infty$ when $u \in  \mathcal E^1_\phi(\Omega)$. Indeed  for $1 \leq p \leq n-1$, we have 
$$
\int_{\Omega} (\phi-u) (dd^c u)^p\wedge (dd^c \phi)^{n-p} \leq  \int_{\Omega} (\phi-u) (dd^c u)^n < \infty,
$$
and then 

$$
\frac{1}{n +1} \int_{\Omega} (\phi-u) (dd^c u)^n \leq E_\phi (u) \leq \int_{\Omega} (\phi-u) (dd^c u)^n.
$$
It turns out that the functionnal $- E_\phi $ is the primitive of the Monge-Ampère operator on the space $\mathcal E^1_\phi(\Omega)$ satisfying $E_\phi(\phi) = 0$.
\begin{lemma}
	Let $u,v \in \mathcal E^1_\phi(\Omega)$ and $u_t = t u + (1-t)v$ for $t \in [0, 1]$. Then $u_t \in \mathcal E^1_\phi(\Omega)$ and we have for $t \in [0,1]$,
 \begin{equation} \label{first derivative}
 \frac{dE_{\phi}(u_t)}{dt}= \int_{\Omega} (v-u)(dd^c u_t)^n,
 \end{equation}
 and 
  \begin{equation} \label{second derivative}
 \frac{d^2E_{\phi}(u_t)}{dt^2}= \int_{\Omega} (v-u) dd^c (u - v)\wedge (dd^c u_t)^{n-1}. 
 \end{equation}
\end{lemma}
\begin{proof}
By Lemma \ref{lem: convexity} we already know that $u_t \in \mathcal E^1_\phi(\Omega)$ for all $t \in [0,1]$.
 We now compute the full energy of $u_t$. For any $t \in [0,1]$, we have
 \begin{flalign*}
  (n+1)E_{\phi}(u_t) & =\int_{\Omega} (\phi - u_t) \sum_{k=1}^n (dd^c u_t)^k \wedge (dd^c \phi)^{n-k},	
	\end{flalign*}
	which is polynomial in $t$. 
 %$$= \sum_{k=1}^n \int_{\Omega}\big(t(\phi-u)+ (1-t)(\phi-v) \big) \sum_{\ell = 0}^k \binom{k}{l} t^l(1-t)^{k-l} (dd^c u)^l \wedge (dd^c v)^{k-l} \wedge(dd^c \phi)^{n-k}.$$
Taking the derivative we obtain	 for $t \in [0,1]$
 \begin{flalign*}
 (n+1) \frac{dE_{\phi}(u_t)}{dt} & = \sum_{k=1}^n \int_{\Omega} (v-u)(dd^c u_t)^k \wedge (dd^c \phi)^{n-k} \\
 &+  \sum_{k=1}^n k \int_{\Omega} (\phi-u_t)(dd^c \dot{u_t}) \wedge (dd^c u_t)^{k-1} \wedge(dd^c \phi)^{n-k}\\
 &= \sum_{k=0}^n \int_{\Omega} (v-u)(dd^c u_t)^k \wedge (dd^c \phi)^{n-k} \\
 & +  \sum_{k=0}^n k\int_{\Omega} (u-v)dd^c(\phi-u_t) \wedge (dd^c u_t)^{k-1} \wedge(dd^c \phi)^{n-k}.
 \end{flalign*}
We have used integration by parts in $\mathcal E^1_\phi(\Omega)$ (see Theorem \ref{thm: integration by parts}).
Therefore 
 \begin{flalign*}
  (n+1) \frac{dE_{\phi}(u_t)}{dt}  &= \sum_{k=1}^n \int_{\Omega} (v-u)(dd^c u_t)^k \wedge (dd^c \phi)^{n-k} \\
  & +  \sum_{k=1}^n k\int_{\Omega} (v-u)(dd^cu_t-dd^c\phi) \wedge (dd^c u_t)^{k-1} \wedge(dd^c \phi)^{n-k}\\
 &= \sum_{k=1}^n (k+1)\int_{\Omega} (v-u)(dd^c u_t)^k \wedge (dd^c \phi)^{n-k} \\
 & - \sum_{k=1}^{n} k \int_{\Omega} (v-u)(dd^c u_t)^{k - 1} \wedge(dd^c \phi)^{n-k +1}\\
 &=  (n+1)\int_{\Omega} (v-u)(dd^c u_t)^n.
 \end{flalign*}
 This prove the formula \eqref{first derivative}.

 Taking the second derivative we obtain
 $$
 \frac{d^2E_{\phi}(u_t)}{dt^2} =  \int_{\Omega} (v-u) dd^c (u-v) \wedge (dd^c u_t)^{n-1}.
$$
\end{proof}
As a consequence of this Lemma, we obtain the following  result :
\begin{corollary} \label{monotonicity-convexity}
 Let $u,v \in \mathcal E^1_\phi(\Omega)$ and $u_t := (1-t)u + t v$ for $0\leq t \leq 1$. Then $u_t \in \mathcal E^1_\phi(\Omega)$ and the function $t \longmapsto E_\phi(u_t)$ is convex in $[0,1]$.
	
 Moreover if  $u \leqslant v$. Then $ E_{\phi}(u) \geqslant E_{\phi}(v).$
\end{corollary}
\begin{proof} The  convexity follows the formula \eqref{second derivative} and Proposition \ref{prop:positivity}. The monotonicity follows from  the formula \eqref{first derivative}.
\end{proof}
\begin{corollary}\label{cor: convexity of energy}
	Let $u,v \in \mathcal E^1_\phi(\Omega)$. Then
	$$ \int_{\Omega} (v-u)(dd^c v)^n
	\leqslant E_{\phi}(u) - E_{\phi}(v) \leqslant  \int_{\Omega} (v-u)(dd^c u)^n.$$
\end{corollary}
\begin{proof} This is a consequence of the convexity the $\phi$-energy along the segments i.e. the derivative of the  convex function  $t \longmapsto E_\phi((1-vt) v + t u)$ is a non decreasing function
in $[0,1]$ (see \eqref{first derivative}).
\end{proof}

\subsection{Continuity of the energy functional}

\begin{lemma}
	\label{lem: bounded energy}
	Let $(u_j) \subset \mathcal{E}^1_\phi(\Omega)$ be a sequence converging to $u \in {\rm PSH}(\Omega)$ in $L^1_{loc}(\Omega,dV)$. Assume that $\sup_j E_\phi(u_j) \leq C< +\infty$. Then $u\in \mathcal E^1_{\phi}(\Omega)$ and $E_{\phi}(u)\leq C$. 
\end{lemma}
\begin{proof}
Define $v_j:= (\sup_{k\geq j} u_k)^*$ for $j \in \N$. Then $v_j \searrow u$ as $j \nearrow + \infty$.  By monotonicity of $E_{\phi}$ we have $E_{\phi}(v_j) \leq E_{\phi}(u_j)$ for any $j\in\N$. 
Therefore, we can assume that $u_j \searrow u$  as $j \nearrow + \infty$. 

Define for any fixed $j \in \N$,  $h_j:= P(u_j-\phi)$. Using $(dd^c h_j)^n \leq {\bf 1}_{\{h_j=u_j-\phi\}}(dd^c u_j)^n$ we have 
	\[
	\int_{\Omega} (-h_j) (dd^c h_j)^n  \leq \int_{\Omega} (\phi-u_j) (dd^c u_j)^n \leq (n+1) E_{\phi}(u_j)\leq (n+1) C, 
	\]
	Therefore, $h:=\lim_j h_j$ belongs to $\mathcal E^1(\Omega)$, hence $P(u-\phi) \in \mathcal E^1(\Omega)$. It remains to prove that 
	\[
	E_{\phi}(u) \leq C. 
	\]
	We first assume that $u-\phi$ is bounded. Then, for all $0\leq p\leq n$, by Proposition \ref{prop: convergence MA quasi-cont} we have the following weak convergence
\[
(\phi-u_j)(dd^c u_j)^p\wedge(dd^c \phi)^{n-p} \rightarrow (\phi-u)(dd^c u)^p\wedge (dd^c \phi)^{n-p}.
\] 
We thus have 
\[
\int_{\Omega}(\phi-u)(dd^c u)^p\wedge (dd^c \phi)^{n-p} \leq \liminf_j \int_{\Omega}(\phi-u_j)(dd^c u_j)^p\wedge(dd^c \phi)^{n-p},
\]
which implies
\[
E_{\phi}(u) \leq \liminf_{j} E_{\phi}(u_j) \leq C. 
\]
	To treat the general case, we set for $t >0$ and $j \in \N$, $v_{t,j}:=\max\{u,u_j,\phi-t\}\in \mathcal E^1_{\phi}(\Omega)$. Since $E_{\phi}$ is non-decreasing we have for any $t >0$ and $j \in \N$,
 \[
 E_{\phi}(v_{t,j}) \leq E_{\phi}(u_j)\leq C. 
 \]
 Fixing $t > 0$ and letting $j\to +\infty$, we deduce from the first case that $E_{\phi}(\max(u,\phi-t))\leq C$. From this and the pluripotential maximum principle we obtain for any $t >0$,
 \[
 \sum_{p=0}^n\int_{\{u>\phi-t\}}(\phi-u)(dd^c u)^p \wedge (dd^c \phi)^{n-p} \leq (n+1)C. 
 \]	
 Letting $t\to +\infty$, we obtain $E_{\phi}(u)\leq C$, finishing the proof. 
\end{proof}

\begin{theorem}\label{Continuity energy}
	Let $(u_j) \subset \mathcal{E}^1_\phi(\Omega)$ be a sequence converging to $u \in {\rm PSH}(\Omega)$ in capacity in $\Omega$. Assume that $u_j\geq v$, for some $v\in \mathcal E^1_{\phi}(\Omega)$. Then, for all $p$, 
	$$
	\lim_{j \to + \infty} \int_{\Omega}(\phi -u_j)(dd^cu_j)^p\wedge (dd^c \phi)^{n-p} = \int_{\Omega}(\phi -u)(dd^cu)^p \wedge (dd^c \phi)^{n-p}.
	$$
\end{theorem}
\begin{proof}
Using the lower bound $u_j\geq v$, and arguing as in Step 2 of the proof of Theorem \ref{thm: integration by parts}, we can assume that the functions $u_j$ are relatively bounded with respect to $\phi$: $u_j\geq \phi-C$. 

Let $(\Omega_k)$ be an exhausting sequence of relatively compact open subsets of $\Omega$, and define $w^k:= P({\bf 1}_{\Omega\setminus \Omega_k} w)$ for a given function $v$. Observe that $w^k=w$ in $\Omega\setminus \bar \Omega_k$; and also $u^k_j \geq v^k$, $u^k \geq v^k$,  $v^k \nearrow \phi$.

	Set $I_j :=  \int_{\Omega}(\phi -u_j)(dd^cu_j)^p\wedge (dd^c \phi)^{n-p}$ and $I:= \int_{\Omega}(\phi -u)(dd^c u)^p\wedge (dd^c \phi)^{n-p}$. Then we have
 \begin{flalign*}
I_j & =  \int_{\Omega}(\phi-u_j^k+ u_j^k -u_j)(dd^cu_j)^p\wedge (dd^c \phi)^{n-p} \\
& = \int_{\Omega}(\phi -u_j^k)(dd^cu_j)^p\wedge (dd^c \phi)^{n-p}\\ &+ \int_{\Omega}(u_j^k -u_j)(dd^cu_j)^p\wedge (dd^c \phi)^{n-p}.
	\end{flalign*}
 Observe that for fixed $k \in \N$, the sequence $j \longmapsto f_j := u_j^k -u_j = (u_j^k - \phi) + (\phi -u_j)$ is a uniformly bounded sequence  of  quasi-continuous function in $\Omega$ with  support contained in a fixed compact set $\bar \Omega_k \Subset \Omega$. Moreover  the sequence $(f_j)$ converges in capacity to the bounded compactly supported quasi-continuous function $f := u^k-u$.  It follows then from  
  Theorem \ref{prop: convergence MA quasi-cont} that for each fixed $k \in \N$, 
 \begin{equation} \label{eq:convergencek}
 \lim_{j \to + \infty} \int_{\Omega}(u_j^k -u_j)(dd^cu_j)^p\wedge (dd^c \phi)^{n-p} = \int_{\Omega}(u^k -u)(dd^cu)^p\wedge (dd^c \phi)^{n-p}.
\end{equation} 
 On the other hand, setting $w_j := (u_j + \phi)\slash 2$ and using Lemma \ref{lem: mixed energy} we have 
 \begin{flalign*}
 	\int_{\Omega}(\phi-u_j^k)(dd^c u_j)^p\wedge (dd^c \phi)^{n-p} & \leq C(n,p)\int_{\Omega}(\phi-u_j^k)(dd^c w_j)^n \leq C E_{\phi}(u_j^k) \leq C E_{\phi}(v^k). 
 \end{flalign*}
 By Lemma \ref{lem: energy convergence increasing} the last term converges to $0$, as $k\to +\infty$, finishing the proof. 
\end{proof}
\subsection{Compactness}
We want to show that the  sublevel sets of the $\phi$-energy functional defined by 
$$
\mathcal E^1_{\phi,C} := \{u \in \mathcal E^1_{\phi}(\Omega) ; E_\phi(u) \leq C\}, \, \, C > 0,
$$
are compact.

We need a Blocki type inequality
\begin{lem} \label{lem:Blocki type inequality}
Let $\psi \in PSH(\Omega) \cap L^{\infty}(\Omega)$. Then we have for any $v \in  \mathcal E^1_{\phi}(\Omega)$
$$
\int_\Omega (\phi -v)^{n+1} (dd^c \psi)^n \leq (n+1)! \sup_{\Omega} |\psi|^n \int_\Omega (\phi-v) (dd^c v)^n.
$$
\end{lem}
\begin{proof} 
It follows from Blocki's inequality \cite{Bl93} and Theorem 3.8 in \cite{Ceg98} that 
\[
\int_{\Omega} (-w)^{n+1} (dd^c \psi)^n \leq (n+1)! \sup_{\Omega} |\psi|^n \int_{\Omega} (-w)(dd^c w)^n, \; w \in \mathcal E^1(\Omega).
\]
Applying this for $w=P(v-\phi)$, and noting that $(dd^c w)^n \leq {\bf 1}_{\{w=v-\phi\}}(dd^c v)^n$, we obtain  
\begin{flalign*}
\int_\Omega (\phi -v)^{n+1} (dd^c \psi)^n &\leq  \int_{\Omega} (-w)^{n+1} (dd^c \psi)^n\\
& \leq (n+1)! \sup_{\Omega} (-\psi)^n \int_\Omega (-w) (dd^c w)^n\\
& \leq (n+1)! \sup_{\Omega} (-\psi)^n \int_\Omega (\phi-v) (dd^c v)^n.	
\end{flalign*}
\end{proof}

\begin{theorem} \label{thm:compactness} For any $C >0$ the set $ \mathcal E^1_{\phi,C} $ is a compact convex subset of $PSH(\Omega)$ for the $L^1_{loc}(\Omega)$-topology.
Moreover the $\phi$-energy functional $E_\phi$ is lower semi-continuous in  $\mathcal E^1_{\phi,C}$.
\end{theorem}
\begin{proof} 
Let $(u_j)_{j \in \N}$ be a sequence  in $ \mathcal E^1_{\phi,C} $ i.e. for any $j\in \mathbb N$,  $E_{\phi}(u_j) \leq C$. By \cite{Kol98}, \cite{Ceg98}, we can find  $\rho \in \mathcal E^0(\Omega)$ solve $(dd^c \rho)^n = dV$. 

By Lemma \ref{lem:Blocki type inequality}  we have 
$$
\int_\Omega (\phi -u_j)^{n+1} (dd^c \rho)^n \leq n! \sup_{\Omega}| \rho|^n \int_\Omega (\phi-u_j) (dd^c u_j)^n.
$$
 Hence the sequence $(\phi - u_j)$ in bounded in $L^1 (\Omega)$. Since $\phi \in L^1_{loc}(\Omega)$,  the sequence $(u_j)_{j \in \N}$ is bounded in $L^1_{loc} (\Omega)$.  Then up to extracting a subsequence we can assume that  $u_j \to u$  in $L^1_{loc}(\Omega)$ and almost everywhere in $\Omega$.
 Now set for $j \in \N$, $\widehat u_j  := (\sup_{k \geq j} u_k)^*$. Then the sequence $(\widehat u_j)_{j\in \N}$ is non increasing and converges  to some function  $u.$ Since $\phi \geq \widehat u_j  \geq u_j$ for any $j\in \N$, it follows that $\widehat u_j   \in \mathcal E^1_{\phi} (\Omega)$ and $E_\phi(\widehat u_j) \leq C$. 
 By Lemma \ref{lem: bounded energy} we conclude that $u \in \mathcal E^1_{\phi}(\Omega)$ and 
 $E_\phi(u) = \lim_j E_\phi (\widehat u_j) \leq C$. Hence $u \in  \mathcal E^1_{\phi,C} $, which proves the compactness.
 
 Now assume that $(u_j)$ is a sequence in $ \mathcal E^1_{\phi,C} $ which converges to $u \in  \mathcal E^1_{\phi,C} $ in $L^1_{loc}(\Omega)$.
 Arguing as in the previous proof, up to substracting a subsequence, we can assume that the sequence $(u_{j})$ converges to $u$ a.e. in $\Omega$ and then 
 $$
 E_\phi(u) = \lim_{j \to + \infty} E_\phi(\widehat u_{j})  \leq \liminf_{j \to + \infty} E_\phi(u_{j}).  
 $$
 This inequality is true for any subsequence of $(u_j)$ proving the semi-continuity of $E_\phi$ in $ \mathcal E^1_{\phi,C}$.
\end{proof}
\section{The variational approach}\label{sect: variational approach}
Let $\mu$ be a positive Borel measure on $\Omega$ satisfying the integrability condition
 $$
\mathcal E^1_{\phi}(\Omega) \subset \phi + L^1(\Omega,\mu).
$$ 
Then we define the integral functional $I_\mu = I_{\mu,\phi} : \mathcal E^1_{\phi}(\Omega) \longrightarrow  \R^+$ by the following formula :
$$
I_\mu (u)   := \int_\Omega (u - \phi) d \mu, \, \, u \in \mathcal E^1_{\phi}(\Omega).
$$ 
We finally define the   functional $F_{\mu} = F_{\mu,\phi} : \mathcal E^1_{\phi}(\Omega) \longrightarrow  \R$ by
\[
F_{\mu}(u) := E_{\phi}(u) +  I_\mu (u) = \frac{1}{n+1} \sum_{p=0}^n\int_{\Omega} (\phi-u) (ddd^c u)^p\wedge (dd^c \phi)^{n-p} + \int_{\Omega} (u-\phi) d\mu. 
\]

At least formally the Euler-Lagrange equation of the functional $F_\mu$ is exactly the Monge-Ampère equation $(dd^c u)^n - \mu = 0$ we are interested in.

\subsection{Coercivity}
\begin{lemma}\label{lem:quantitative integrabilily}
Let $\mu$ be a positive Borel measure such that $\mathcal E^1_{\phi}(\Omega) \subset \phi+   L^1(\Omega,\mu)$. Then there exists a constant $C>0$ depending on $\mu$, such that 
\[
\int_{\Omega}(\phi-u) d\mu \leq C (E_{\phi}(u)^{1/2}+1), \; u\in \mathcal E^1_{\phi}(\Omega). 
\]
\end{lemma}
\begin{proof}
Assume by contradiction that there exists a sequence $(u_j)$ in $\mathcal E^1_{\phi}(\Omega)$  with 
\[
\int_{\Omega} (\phi-u_j) d\mu \geq 4^j (E_{\phi}(u_j)^{1/2}+1). 
\]	
We set for $j \in \N$, $m_j^{-1} =E_\phi(u_j)^{\frac{1}{2}}+1$ and define 
	$$
	v_j = m_ju_j + (1 - m_j)\phi.
	$$
 By convexity $v_j \in \mathcal E^1_{\phi}(\Omega)$ for any $j\in \N$. 
 
  We claim that there exists a constant $c_n > 0$ such that for any $j \in\N$,  $E_\phi(v_j) \leqslant c_n$. Indeed, using  the fact that $(dd^c \phi)^n=0$,  we have 
 \begin{flalign*}
 E_{\phi}(v_j) &\leq \int_{\Omega} (\phi-v_j) (dd^c v_j)^n \\
 &	 =m_j \int_{\Omega} (\phi-u_j) \sum_{p=1}^n m_j^{p} (1-m_j)^{n-p} \binom{n}{p}   
 (dd^c u_j)^p \wedge (dd^c \phi)^{n-p} \\
 & \leq c_n  m_j^2 E_{\phi}(u_j) \leq c_n. 
 \end{flalign*}
 We next define  for $k \in \N^*$,
 $$
 w_k := \sum_{j=1}^{k} 2^{-j}v_j + \sum_{j=k+1}^{+\infty} 2^{-j} \phi.
 $$
 By convexity of the energy $E_{\phi}$, we have $w_k \in \mathcal{E}_\phi^1(\Omega)$ and $E_\phi(w_k) \leq c_n$ for any $k \in \N$. Since the sequence $(w_k)$ is non increasing, it follows from Lemma \ref{lem: bounded energy} that its limit $ w \in \mathcal{E}_\phi^1(\Omega)$. Moreover, 
	\begin{flalign*}
		\int_{\Omega}(\phi - w_k)d\mu & = \sum_{j=1}^{k}2^{-j} \int_{\Omega} (\phi - v_j)d\mu\\
		& = \sum_{j=1}^{k}2^{-j}m_j \int_{\Omega} (\phi - u_j)d\mu\\
		& \geqslant \sum_{j=1}^{k}2^j= 2(-1 + 2^k) \longrightarrow + \infty,
	\end{flalign*}
	so $ \int_{\Omega}(\phi - w)d\mu = + \infty$ which is a contradiction because $w \in \mathcal{E}_\phi^1(\Omega)$. 
\end{proof}

We now prove the properness of the functional $F_{\mu}$.
\begin{prop}\label{prop:coercivity}
Let $\mu$ be a positive Borel measure on $\Omega$  such that $\mathcal E^1_\phi(\Omega) \subset \phi+ L^1(\mu)$. Then the functional $F_\mu$ is coercive.
\end{prop}
\begin{proof}
By Lemma \ref{lem:quantitative integrabilily}, there exists a constant $C >0$ such that   for all $u\in \mathcal E^1_{\phi}(\Omega)$,
\[
\int_{\Omega}(\phi-u) d\mu \leq C (E_{\phi}(u)^{1/2}+1). 
\]
	Then by enlarging the constant $C$ we obtain a constant $C'> 0$ such that 
\begin{align*}
	F_\mu(u)&=E_{\phi}(u) + \int_{\Omega} (u-\phi) d\mu\\&\geqslant E_{\phi}(u)-C (E_{\phi}(u)^{1/2}+1) \geq \frac{1}{2} E_{\phi}(u) -C'. 
\end{align*}
\end{proof}

\begin{lemma} \label{lem:existence minimizer}
	Assume $I_\mu : \mathcal{E}_{\phi}^1(\Omega) \longrightarrow \R$ is continuous. Then  there exists $u \in \mathcal{E}_{\phi}^1(\Omega)$ such that
	 $$
	 F_\mu(u) = \inf_{v \in \mathcal{E}_{\phi}^1(\Omega)} F_\mu(v).
	 $$
\end{lemma}
\begin{proof} By Lemma \ref{prop:coercivity} there exists a constant $B > 0$ such that  
$F_\mu (u) \geq (1\slash 2) E_\phi(u) - B \geq -B $ for any $u \in \mathcal{E}_{\phi}^1(\Omega)$. Then  $m := \inf_{v \in \mathcal{E}_{\phi}^1(\Omega)} F_\mu(v) \geq -B$. 
Moreover since $F_\mu(u) \to + \infty$ as $E_\phi(u) \to + \infty$, there exists $C >0$ such that for any 
 $$
 F_\mu (u) > m+1, \, \, \text{ if } u \in \mathcal{E}_{\phi}^1(\Omega)\, \, \text{and} \, \,  E_\phi(u) \geq C.
 $$
Therefore 
$$
m = \inf_{v \in \mathcal{E}_{\phi,C}} F_\mu(v).
$$
Recall that $F_\mu = E_\phi - I_\mu$. If $I_\mu$ is continuous then $F_\mu$ is lower semicontinuous by Theorem \ref{thm:compactness}.
Since   $\mathcal{E}_{\phi,C}$ is compact by Theorem \ref{thm:compactness}, it follows that$F_\mu$ achieves its minimum at some $u \in \mathcal{E}_{\phi,C}$.
\end{proof}
\subsection{Continuity of the functional $I_\mu$}
It remains to investigate the continuity of $I_\mu$.

\begin{lemma} \label{continuity I mu}
Assume that $\mu := (dd^c \psi)^n$ where $\psi \in PSH(\Omega)\cap L^{\infty}(\Omega)$ and $\int_{\Omega}(dd^c \psi)^n<+\infty$. 
Then the functional $I_\mu$ is continuous on each $\mathcal E^1_{\phi,C}$.
\end{lemma} 
\begin{proof} 
We use Cegrell's trick based on Banach-Sacks theorem.
Let $(u_j)$ be sequence in $\mathcal{E}^1_{\phi,C}$ which converges to $u \in \mathcal{E}^1_{\phi,C}$ in the $L^1_{loc}(\Omega)$-topology. By Lemma \ref{lem:Blocki type inequality} and H\"older's inequality, $(u_j-\phi)_{j\in \N}$  is bounded in $L^2(\Omega,\mu)$.  We assume that $\lim_j \int_\Omega (u_j - \phi) d \mu = a$ exists and the goal is to prove that $a=\int_{\Omega}(u-\phi)d\mu$. 

Taking a subsequence we can assume that $u_j \to u$ almost everywhere with respect to the Lebesgue measure.  By Banach-Sacks Theorem, after extracting a subsequence, still denoted by $(u_j)$, there exists $v\in L^2(\Omega,\mu)$ such that 
 $$
\frac{1}{k} \sum_{j = 1}^k (u_k-\phi)  \to v  \, \, \text{in} \, \, L^2(\Omega,\mu). 
 $$
 We claim that $u=\phi+v$ a.e. with respect to $\mu$. If the claim is proved then 
 $$
\int_\Omega(u-\phi) d  \mu= \int_\Omega v d \mu = \lim_k \left(\frac{1}{k} \sum_{j=1}^k \int_\Omega (u_j-\phi) d \mu\right)=a,
$$
finishing the proof. 
To prove the claim, we define 
 \[
 w_j := \frac{1}{j} \sum_{k = 1}^j u_k. 
 \] 
  Since $w_j\to v+\phi$ in $L^2(\Omega,\mu)$, there is a subsequence $(w_{l_j})$ such that $w_{l_j} \to v+\phi$ a.e. with respect to $\mu$. We next define 
\[
 \psi_j:= \left(\sup_{k\geq j} w_{l_k}\right)^*. 
 \]
  
  Since $w_{l_j} \to u$ in $L^1_{loc}(\Omega)$, it follows that   $\psi_j \searrow u$  in $\Omega$. 
We also have $w_{l_j} \to v+\phi$ a.e. with respect to $\mu$ and $\mu$ vanishes on the pluripolar set $\{\sup_{k \geq j} w_{j_k}<\psi_{j}\}$. It thus follows that $\psi_{j}$ converges to $v+\phi$ $\mu$-a.e. Therefore,  $\psi_{j_k} \to u$ $\mu$-a.e. proving the claim. 
\end{proof}

\subsection{Minimizers are solutions}
We will show that a minimizer of the functional $F_\mu$ is a solution to the Monge-Ampère equation in $\mathcal E^1_\phi (\Omega)$.

\begin{lemma}\label{lem: proj is in E1 phi}
	If $u\in \mathcal E^1_{\phi}(\Omega)$ and $\psi$ is a continuous function with compact support in $\Omega$, then $P(u+\psi) \in \mathcal E^1_{\phi}(\Omega)$. 
\end{lemma}

\begin{proof}
	Fix $h\in  \mathcal E^1(\Omega)$, such that $h<0$ and $\phi+h\leq u$. Multiplying $h$ with a positive constant we can assume that $h\leq \psi$ on the support $K$ of $\psi$, hence on $\Omega$. It follows that $u+h \leq u + \psi$ in $\Omega$. This proves that $v:=P(u+\psi)$ is well defined and satisfies 
	\[
	\phi \geq v \geq u+h\geq \phi+ P(u-\phi)+h. 
	\]
	Since $P(u-\phi)+h\in \mathcal{E}^1(\Omega)\subset \mathcal N^a(\Omega)$, it remains to prove that 
	\begin{equation}
		\label{eq: proj in E1phi}
		\int_{\Omega}(\phi-v) (dd^c v)^n<+\infty. 
	\end{equation}
	We first assume that $u-\phi$ is bounded. Recall that $(dd^c v)^n$ is supported on the contact set $\{v=u+\psi\}$. Therefore,
	\[
	\int_{K} (\phi-v) (dd^c v)^n = \int_{K}(\phi-u) (dd^c v)^n + \int_{K}(- \psi) (dd^c v)^n  <+\infty, 
	\]
	because $(dd^c v)^n$ has finite mass on compact sets. On the other hand, in the open set $\Omega\setminus K$, $v\leq u$ and $(dd^c v)^n$ is concentrated on the contact set $\Omega\setminus K \cap \{v=u+\phi=u\}$. The pluripotential maximum principle then yields 
	\[
	{\bf 1}_{\Omega\setminus K}(dd^c v)^n = {\bf 1}_{(\Omega\setminus K) \cap \{u=v\}}(dd^c v)^n \leq {\bf 1}_{(\Omega\setminus K) \cap \{u=v\}} (dd^c u)^n. 
	\]
	We therefore get 
	\[
	\int_{\Omega \setminus K} (\phi-v) (dd^c v)^n \leq \int_{\Omega \setminus K} (\phi-u) (dd^c u)^n <+\infty. 
	\]
	Thus, \eqref{eq: proj in E1phi} is proved for the case when $u-\phi$ is bounded. 
	
	We now treat the general case. For $t>0$ we define $u_t:= \max(u,\phi-t)$, $v_t:=P(u_t+\psi)$.  By the previous step we have $v_t \in \mathcal E^1_{\phi}(\Omega)$. By convexity of $E_{\phi}$, we also have 
	\[
	E_{\phi}(v_t)-E_{\phi}(u_t) \leq \int_{\Omega}(u_t-v_t)(dd^c v_t)^n \leq  \int_{\Omega} (-\psi) (dd^c v_t)^n.
	\]
	Letting $t\to +\infty$, we see that 
	\[
	\sup_t E_{\phi}(v_t) \leq E_{\phi}(u)+ \int_K (-\psi) (dd^c v)^n <+\infty. 
	\]
By Lemma \ref{lem: bounded energy} we thus have $v\in \mathcal E^1_{\phi}(\Omega)$. 
\end{proof}
From the definition it follows that if $u \in \mathcal E_\phi (\Omega)$, then $P_\Omega (u-\phi) \in \mathcal E^1(\Omega)$ and $ P_\Omega(u-\phi) + \phi \leq u \leq \phi$. When $n = 1$, $P_\Omega (u-\phi) = G_\Omega (\mu_u)$ is the Green potential of the Riesz measure  $\mu :=  (1 \slash 2\pi) \Delta u$ and we have equality $ u = \phi +  P(u-\phi) = \phi + G_\Omega (\mu_u)$. This is known as the Riesz representation formula.

 The following example shows that  there is not such representation  formula in higher dimension,  proving that we cannot drop the finiteness energy condition $e_\phi(u) < + \infty$ in the definition of the class $\mathcal E_\phi (\Omega) $.
\begin{example}\label{example: E1}
 Let $\Omega=\mathbb B^{n}$ be the unit ball in $\mathbb C^{n}$ and define 
 \[
 \phi(z_1,...,z_n) = -(1-|z_1|^2-...-|z_{n-1}|^2)^{\alpha}, \, \, z \in \mathbb B^{n}, \, \, 0<\alpha<1.
 \]
  Then $\phi$ is a bounded maximal plurisubharmonic function in $\Omega$ hence  
  $\phi \in \mathcal M(\B^n)$. 
 
 We claim that for $n\geq 3$, we can find $v\in  \mathcal E^1(\Omega)$ such that $u := v + \phi \notin \mathcal E^1_{\phi}(\Omega)$.
 
 Let $v$ be the psh function in $\Omega$ defined by 
	\[
	v (z):= -(1-\|z\|^2)^{\gamma},
	\]
	 where $0<\gamma<1$. We compute  
	 \begin{flalign*}
	 	 dd^c v = \gamma (1-\|z\|^2)^{\gamma-1} \sum_{j} i dz_j \wedge d\bar z_j + \gamma(1-\gamma) (1-\|z\|^2)^{\gamma-2} d \|z\|^2 \wedge d^c \|z\|^2. 
	 \end{flalign*}
	 Thus, 
	 \[
	 (dd^c v)^{n} \sim   (1-\|z\|^2)^{n \gamma -n-1}dV(z). 
	 \]
 Thus $v \in \mathcal E^1(\Omega)$ if and only if $(n+1)\gamma -(n+1) >-1$, which is equivalent to $\gamma > n/(n+1)$. 

 We compute similarly 
 \[
 (dd^c \phi)^{n-1} \sim (1-|z_1|^2-...-|z_{n-1}|^2)^{(n-1)\alpha -n} \wedge_{j=1}^{n-1} idz_j \wedge d\bar z_j,
  \]
 and 
 \begin{flalign*}
 \int_{\Omega} (-v) (dd^c (\phi +v))^n & \geq \int_{\Omega} (-v) (dd^c \phi)^{n-1} \wedge dd^c v \\
 &\geq \int_{\Omega} (1-\|z\|^2)^{2\gamma-1} (1-\|z\|^2+|z_n|^2)^{(n-1)\alpha -n} dV(z).
 \end{flalign*}
	The last integral is $+\infty$ when $2\gamma + (n-1)\alpha -n\leq -1$. Thus, for $n\geq 3$, we can find $v\in \mathcal E^1(\Omega)$ such that $\phi+v \notin \mathcal E^1_{\phi}(\Omega)$. 
	
		For $n\geq 4$, $\gamma=1$, $(n-1)\alpha \leq n-3$, we see that $u\in \mathcal E^0(\Omega)$ but $u+\phi \notin \mathcal E^1_{\phi}(\Omega)$. 

\end{example}

\begin{lemma}\label{DerivEnergie}
	Let $u \in \mathcal{E}_{\phi}^1(\Omega)$ and  $\psi \in  C_c^0(\Omega)$ be a continuous function with compact support in $\Omega$. Then
	$$\frac{dE_{\phi} \circ P (u + t\psi)}{dt}_{\mid t = 0}= \int_{\Omega} (-\psi)(dd^c u)^n.$$
\end{lemma}
\begin{proof}
	For $t>0$, the convexity of $E_{\phi}$, see Corollary \ref{cor: convexity of energy}, gives 
	$$\int_{\Omega} \frac{u-P(u+t\psi)}{t} (dd^c u)^n  \leqslant \frac{E_{\phi}(P(u+t\psi))- E_{\phi}(u)}{t} \leqslant  \int_{\Omega} \frac{u-P(u+t\psi) }{t}(dd^c P(u+t\psi))^n.
	$$
	Since $(dd^c P(u+t\psi))^n$ is supported on $\{P(u+t\psi)=u+t\psi\}$, we can continue the above inequalities as 
	$$\int_{\Omega} (-\psi) (dd^c u)^n \leqslant \frac{E_{\phi}(P(u+t\psi))- E_{\phi}(u)}{t} \leqslant  \int_{\Omega} (-\psi) (dd^c P(u+t\psi))^n.
	$$
	Letting $t\to 0^+$, we arrive at the formula for the right-derivative. The same argument works for the left derivative, finishing the proof. 
\end{proof}

 \begin{theorem}\label{thm: minimizer is a solution}
 Assume that  $u \in \mathcal{E}_\phi^1(\Omega)$ is a minimizer of the functional $F_\mu$ i.e. 
 $$
 F_\mu(u) = \inf_{v \in \mathcal{E}_{\phi}^1(\Omega)} F_\mu(v).
 $$ 
 There  $(dd^c u)^n = \mu$ on $\Omega$.
\end{theorem}
\begin{proof}
Assume that $u \in \mathcal{E}_\phi^1(\Omega)$ is a minimizer of $F_\mu$. Fix $\chi \in \mathcal \mathcal C_c^0(\Omega)$ a continuous function with compact support.
By Lemma \ref{lem: proj is in E1 phi}, $P(u + t \chi) \in  \mathcal{E}_\phi^1(\Omega)$.
Define for $t \in \R$,
$$
g(t) := E_\phi\bigl(P(u + t \chi)\bigr) + \int_\Omega (u + t \chi - \phi) \, d\mu, \; t \in \mathbb R. 
$$
Observe that
 since $P(u + t \chi) \leqslant u + t \chi$, we have 
 $$g(t) \geqslant F_\mu(P(u + t\chi)) \geq F_\mu (u) = g(0).
 $$ 
 Hence $g$ achieves its minimum at $t=0$. Since $g'(0)$ exists we have  $g'(0) = 0$. Since the function $t \mapsto \int_\Omega (u-\phi + t \chi) \, d\mu$ is differentiable, with derivative given by $\int_\Omega \chi \, d\mu$,  using Lemma \ref{DerivEnergie}, it follows that
$$
\int_\Omega \chi (dd^c u)^n =  \int_\Omega \chi \, d\mu.
$$ 
Since this is true for all  $\chi \in \mathcal \mathcal C_c^0(\Omega)$, we infer 
 that $(dd^c u)^n =  \mu$ in the sense of Radon measures in $\Omega$. 
\end{proof}

\section{Proofs of the main results}\label{sect: proof of the main results}
\subsection{Proof of Theorem \ref{thmA}}
We prove the following result.
 \begin{theorem}
 Let $\mu$ be a non pluripolaire positive Borel measure on $\Omega$. Then the following properties are equivalent:
	\begin{enumerate}
 \item  $\mathcal{E}_\phi^1(\Omega) \subset \phi + L^1(\Omega,\mu); $
 \item There exists a constant $C >0$ such that for any $v \in \mathcal{E}_\phi^1(\Omega)$,
 $$
  \int_{\Omega} (\phi-v)d\mu  \leqslant C \max(E_\phi(v)^{\frac{1}{2}}+ 1)
  $$
 \item  there exists $u \in \mathcal{E}_\phi^1(\Omega)$ such that $(dd^c u)^n = \mu$.
	\end{enumerate}
	Moreover the solution is unique.
\end{theorem}
\begin{proof}
The implication \textbf{(1) $\Rightarrow$ (2)} follows from Lemma \ref{lem:quantitative integrabilily}
and  the implication \textbf{(3) $\Rightarrow$ (1)} follows from Lemma \ref{lem: mixed energy inequality}

Let us prove that \textbf{(2) $\Rightarrow$ (3)}.

 We assume first that $\mu$ has a compact support. By the generalized Radon–Nikodym theorem (see \cite[Theorem 6.3]{Ceg98}), there exist $h \in \mathcal{E}^0(\Omega)$  and $f \in L^1((dd^c h)^n)$ such that  
$$
 \mu = f (dd^c h)^n.
$$
 Consider the measures  
 $$
 \mu_j := \min(f, j^n) (dd^c h)^n, j \in \N.
 $$  
 Since $\mu_j \leq  (dd^c h_j)^n$ where $h_j := j h \in \mathcal E^0(\Omega)$, it follows from Lemma 
 \ref{continuity I mu}, Lemma \ref{lem:existence minimizer} and Theorem \ref{thm: minimizer is a solution},  that  
 for each $j\in \N$, there exists $u_j \in \mathcal{E}^1_\phi(\Omega)$ such that  $(dd^c u_j)^n = \mu_j$.
 Since $\mu_j \leq \mu$ it follows from $(2)$ that for any $j\in \N$,
 $$
 E_\phi (u_j) = \int_{\Omega} (\phi-u_j)d\mu_j \leq  \int_{\Omega} (\phi-u_j)d\mu  \leqslant C \max(E_\phi(u_j)^{\frac{1}{2}}+ 1),
 $$
 hence $\sup_j E_\phi(u_j) < + \infty$.
 
 Since $(\mu_j)$ is increasing, by the comparison principle, the sequence $(u_j)$ is decreasing,  
 hence $u = \lim_j u_j \in  \mathcal{E}^1_\phi(\Omega)$.
 By the continuity of the Monge-Ampère operator it follows that $(dd^c u)^n = \mu$.

 When $\mu$ has a non compact support, we take an increasing sequence of compact sets $(K_j)$ conveging to $\Omega$ and set $\mu_j := {\bf 1}_{K_j} \mu$.
 By the previous analysis there exsists $u_j \in  \mathcal{E}^1_\phi(\Omega)$ such that
 $(dd^c u_j)^n = \mu_j$ for any $j\in \N$.
 The same analysis as in the previous case shows that $(u_j)$ decreases to $u \in \mathcal{E}^1_\phi(\Omega)$ which satisfies $(dd^c u)^n = \mu$.
	
 The uniqueness follows from Poposition  \ref{prop:uniqueness}.
\end{proof}
We now prove  Corollary \ref{coroB}. 
\begin{proof} Assume that $\mu \leq (dd^c \psi)^n$ for some $\psi\in \mathcal E^1(\Omega)$.
It is enough to prove that $ \mathcal{E}^1_\phi(\Omega) \subset \phi + L^1(\Omega,\mu)$.
Indeed let $u\in  \mathcal{E}^1_\phi(\Omega)$. Then $v:=P(u-\phi) \in \mathcal E^1(\Omega)$  and 
\[
\int_{\Omega} (-v) (dd^c v)^n \leq \int_{\Omega} (\phi-u) (dd^c u)^n  <+\infty. 
\]
Hence by Cegrell \cite{Ceg04}
$$
\int_\Omega (\phi-u) d \mu \leq \int_\Omega (-v) (dd^c \psi)^n \leq E_0(v)^{1 \slash n+1} E_0(\psi)^{n \slash n+1}< +\infty.
$$
\end{proof}
\begin{example} \label{example: finite entropy} Let $ \mu := f d V$ be a positive Borel measure with finite entropy i.e.   
$$
\int_\Omega f \log(1 + f)  d V < +\infty.
$$
We claim that $\mathcal E^1(\Omega) \subset L^1(\Omega,\mu)$.

Indeed let $v \in  \mathcal E^1(\Omega)$. By \cite{ACKPZ09}, \cite{BB22},  for any $\alpha >0$, $\int_\Omega e^{- \alpha v} d V < + \infty$.
Let $\chi(t) := e^t - t -1$ for $t \geq 0$. This is a Young function on $\R^+$ such that 
$\int_\Omega \chi(-v) d V < + \infty$. The conjugate function of $\chi$ is defined by 
$$
\chi^*(s) := \sup_{t \geq 0} \{ st - \chi(t)\} = (s+1) \log (s+1) \simeq_{s \to + \infty} s \log s.
$$
Then $\int_\Omega \chi^*(f) d V < +\infty$. 
Observe that by definition we have $s t \leq \chi(t) + \chi^*(s) $ for any $s, t \geq 0$.
Therefore
$$
\int_\Omega (- v) f dV \leq \int_\Omega \chi(-v) d V + \int_\Omega \chi^*(f) d V < + \infty,
$$
which proves the Claim.
By Theorem \ref{thmA} (see also \cite{Ceg98}) there exists $v \in \mathcal E^1(\Omega)$ such that
$\mu = (dd^c v)^n$. Then by Corollary \ref{coroB} there exists $u \in \mathcal E^1_\phi(\Omega)$ such that
$(dd^c u)^n = \mu$ on $\Omega$.
\end{example}
\subsection{Proof of Theorem \ref{thmC}}
In this section we show that Theorem \ref{thmC} is a reformulation of Theorem \ref{thmA}, once we use the toric correspondence between the real Monge-Ampère operator and the complex Monge-Ampère operator which was already deeply used in \cite{LZ25}. 

Let us recall the setting, referring   the interested readers to \cite{LZ25} for more details. 

Let $D \Subset \R^n$ be a bounded {\it convex} domain. We denote by ${\rm CV}(D)$ the positive cone of real valued convex functions in $D$, and by ${\rm CV}_0(D)$ the positive cone of  real valued convex functions in $D$ that are continuous on $\bar D$ and vanish on the boundary $\partial D$.  

Define $\Omega=L^{-1}(D)$, where $L$ is the Log map defined on $(\mathbb C^*)^n$ by 
\[
L(z):= (\log |z_1|,...,\log|z_n|). 
\]
As shown in \cite{LZ25}, $\Omega$ is a bounded hyperconvex domain in $\mathbb C^n$. A defining function for $\Omega$ can be given as 
\[
\rho(z):= - {\rm dist}(L(z),\mathbb R^n \setminus D).
\] 
The real Monge-Amp\`ere measure of a convex function $u$ acting on a Borel set $K$ is defined as 
$$
\int_K \text{M}_\R(u) = \vert \nabla u (K)\vert_n,
$$
where $\vert \cdot \vert_n$ is the $n$-dimensional volume, $\nabla u (K) := \bigcup_{a\in K} \nabla u (a)$, and
$$
  \nabla u (a) := \{ \xi \in \R^n ; u(x) \geq u(a) + \langle \xi, x-a\rangle, \forall x \in D\}. 
$$ 
The following deep connection between the real and complex Monge-Amp\`ere operators was discovered in \cite{BB13} (see also  \cite{CGSZ19}): if $u$ is convex and $\widehat{u}=u\circ L$ then for any continuous test function $\chi$ in $D$, 
\begin{equation} \label{eq:MAR-MAC}
 \int_\Omega \chi \circ L (dd^c \widehat{u})^n = \int_D \chi \mathrm{M}_\R(u).
\end{equation}
This means that $L_*((dd^c \widehat u)^n) = \mathrm{M}_\R(u)$ in the sense of Borel measures on $D$.

Recall the action of the unit torus $\mathbb T^n =(\mathbb S^1)^n$ on $(\C^*)^n$ defined by for $(\zeta,z) \in \mathbb T^n \times (\C^*)^n$ by the formula 
$$
\tau_\zeta(z) = \zeta \cdot z := (\zeta_1 z_1, \cdots, \zeta_n z_n),
$$
which defines a unitary automorphism $ \tau_\zeta : \Omega \longrightarrow \Omega$.

Observe that   $\chi \longrightarrow \widehat \chi := \chi \circ L$ is  a one-to-one correspondance between  continuous test functions on $D$ and  continuous toric test function on $\Omega$.

This induces a one-to-one correspondence between Borel measures on $D$ and toric Borel measure on $\Omega$. 
Recall that a toric Borel measure  on $\Omega$ is a Borel measure which is invariant by this action i.e. $(\tau_\zeta)_*(\mu) = \mu$ for any $\zeta \in \mathbb T^n$.

\begin{lemma}\label{lem:torification}
	If  $\nu$ is a Borel measures on $D$  then there exists a unique toric measure $\mu := \widehat \nu$ on $\Omega$ such that $L_*(\mu) = \nu$ in the sense of Borel measures on $D$. Moreover, $\mu$ vanishes on pluripolar sets. 
\end{lemma}
\begin{proof}
	The condition $L_*(\mu) = \nu$ on $D$ defines $\mu$ in a unique way on continuous  toric test functions 
$\psi$ on $\Omega$ by the formula
$$
\int_\Omega \widehat \chi  d \mu = \int_\Omega \chi \circ L d\mu = \int_D \chi  d\nu,
$$
where $\chi$ any continuous test function on $D$. 

Then we extend $\mu$ to all continuous test functions on $\Omega$ by the formula
$$
\int_\Omega \psi d \mu :=  \int_\Omega  \psi^\tau d \mu,
$$
where 
\begin{equation}\label{eq: torification}
	\psi^\tau (z) :=  \int_{\mathbb T^n} \psi \circ \tau_\zeta(z) d \sigma_{\mathbb T^n} (\zeta)
\end{equation}
is the "torification" of $\psi$.

	We take a pluripolar set $E\Subset \Omega$ and prove that $\mu(E)=0$. We can assume that $\nu(D)<+\infty$ by considering ${\bf 1}_{D'}\nu$, where $E\subset D'\Subset D$. By Josefson's theorem (see \cite{Jos78}, \cite[Theorem 4.40]{GZbook}), we can find $\psi \in {\rm PSH}(\Omega)$, $\psi \leq 0$, such that $E\subset \{\psi=-\infty\}$. The torification of $\psi$ is a plurisubharmonic function which is $\mathbb T^n$-invariant, hence we can write $\psi^{\tau}= v\circ L$, for some convex function $v$. Since $v$ is bounded (because $v\leq 0$ and it is convex) it thus follows that 
	\[
	\int_{\Omega} (-\psi) d\mu = \int_{\Omega} (-\psi^{\tau}) d\mu = \int_D (-v) d\nu <+\infty.
	\]
Therefore $\mu(E)\leq \mu(\{\psi=-\infty\})=0$, as desired.  
\end{proof}

Let $\varphi: D\rightarrow \mathbb R$ be a maximal convex function in $D$. Here, maximal means $\mathrm{M}_{\mathbb R}(\varphi)=0$ in $D$. We stress that we do not assume $\phi$ is bounded in $D$. Then $\phi:=\varphi\circ L$ is a maximal plurisubharmonic function in $\Omega$ which is invariant by the action of $\mathbb T^n$.

\begin{definition}
	A convex function $v\in {\rm CV}_0(D)$ belongs to $\mathcal E^1_0(D)$ if $\int_D (-v)\mathrm{M}_{\mathbb R}(v)<+\infty$. 
	
	We let $\mathcal E^1_{\varphi}(D)$ denote the set of all convex functions $u\in {\rm CV}(D)$ such that 
	\[
	\lim_{z\to \partial D}(u(z)-\varphi(z))=0\; ; \;  e_{\varphi}(u):=\int_{D}(\varphi-u) \mathrm{M}_{\mathbb R}(u)<+\infty.
	\]   
\end{definition} 

As shown in \cite{Ceg98}, any function $u$ in the Cegrell's class $\mathcal E^1(\Omega)$ satisfies 
\[
\limsup_{z \to w} u(z) =0, \; w\in \partial \Omega,
\]
but in general we do not have $\lim_{z \to w} u(z) =0$. A crucial contribution of \cite[Proposition 7.3]{LZ25} is that this holds for those  functions that are $(\mathbb S^1)^n$-invariant. More precisely,  $u\circ L \in \mathcal N_0(\Omega)$ if and only if $u \in \mathcal {\rm CV}_0(D)$. A direct consequence of \cite[Proposition 7.3]{LZ25} is the following. 

\begin{proposition}\label{prop: energy MAR-MAC}
	Let $v\in {\rm CV}(D)$. Then $\hat{v} := v\circ L$ belongs to $\mathcal E^1_{\phi}(\Omega)$ if and only if $v\in \mathcal E^1_{\varphi}(D)$.  In this case we have
 $$
  \int_\Omega (\phi-\hat{v}) (dd^c \hat{v})^n = \int_D (\varphi-v){\rm M}_\R(v).
 $$
\end{proposition}

\begin{lemma}\label{lem: torification in E1}
	If $u\in \mathcal E^1_{\phi}(\Omega)$ then $u^{\tau}\in \mathcal E^1_{\phi}(\Omega)$ and $E_{\phi}(u^{\tau})\leq E_{\phi}(u)$. 
\end{lemma}
\begin{proof}
	Take $u\in \mathcal E^1_{\phi}(\Omega)$ and let $u^{\tau}$ denote its torification function, as defined in \eqref{eq: torification}. We proceed in several steps. 
	
{\bf Step 1.}  Assume that $u-\phi$ is bounded and $\zeta \mapsto u\circ \tau_{\zeta}(z)$ is continuous for each $z$ fixed. The measure $\sigma_{\mathbb T^n}$ is the product of the normalized Lebesgue measure on the circle $\mathbb S^1$, which can be written as the weak limit of equipartitioned Dirac masses. We can thus write 
\[
\sigma_{\mathbb T^n} = \lim_{N\to +\infty}  \frac{1}{N}\sum_{j=1}^n \delta_{\zeta_{N,j}},
\]
where the convergence holds in the weak sense of probability measures on the torus $\mathbb T^n$. 
Since $\zeta \mapsto u(\tau_{\zeta}(z))$ is continuous for each $z$ fixed,  the integral in \eqref{eq: torification} can be written as  
\[
u^{\tau}(z) = \lim_N u^{\tau}_N(z):= \lim_N \frac{1}{N}\sum_{j=1}^N u_{N,j}(z), \; u_{N,j}(z) = u\circ \tau_{\zeta_{N,j}}(z). 
\]
Since $\phi$ is $\tau_{\zeta}$-invariant for any $\zeta$,  we also have 
\[
E_{\phi}(u\circ \tau_{\zeta}) = \frac{1}{n+1} \sum_{p=0}^n \int_{\Omega} (\phi-u)\circ \tau_{\zeta} (dd^c u\circ \tau_{\zeta})^p \wedge (dd^c \phi)^{n-p} =  E_{\phi}(u). 
\]
By convexity of $E_{\phi}$, it follows that $E_{\phi}(u^{\tau}_N) \leq E_{\phi}(u)$. Observe also that $u^{\tau}_N$ converges to $u^{\tau}$ in $L^1_{loc}(\Omega)$.  By Lemma \ref{lem: bounded energy} we thus have $u^{\tau} \in \mathcal E^1_{\phi}(\Omega)$ and $E_{\phi}(u^{\tau})\leq E_{\phi}(u)$. 

{\bf Step 2.} We next assume that $u-\phi$ is bounded but we remove the continuity assumption in the first step. We define 
\[
v_{\varepsilon}(z) := \left (\sup_{|\zeta-1|<\varepsilon} u\circ \tau_{\zeta}(z) \right)^*. 
\]
Then $v_{\varepsilon} \in \mathcal E^1_{\phi}(\Omega)$ and it satisfies the continuity assumption in Step 1. Moreover, $E_{\phi}(v_{\varepsilon}) \leq E_{\phi}(u)$. Since $v_{\varepsilon}\searrow u$, we have $v_{\varepsilon}^{\tau} \searrow u^{\tau}$ and  by Lemma \ref{lem: bounded energy} yields $u^{\tau} \in \mathcal E^1_{\phi}(\Omega)$ with $E_{\phi}(u^{\tau})\leq E_{\phi}(u)$. 

{\bf Step 3.} To treat the general case we approximate $u$ by $v_t:=\max(u,\phi-t$) and apply the previous steps, noting that $E_{\phi}(\max(u,\phi-t)) \leq E_{\phi}(u)$.

\end{proof}

We are now ready to prove Theorem \ref{thmC}. 

\begin{proof}[Proof of Theorem \ref{thmC}]
Using Proposition \ref{prop: energy MAR-MAC} we can  prove the implications $(1) \Rightarrow (2)$ and $(3)\Rightarrow (1)$ exactly as we have done in Theorem \ref{thmA}. 

We provide below two proofs of the implication $(2)\Rightarrow (3)$. The first is obtained by adapting the proof of Theorem \ref{thmA}, while the second relies on showing that condition $(2)$ in Theorem \ref{thmC} implies that the torification measure satisfies condition $(2)$ in Theorem \ref{thmA}.

We explain the first proof. 
Take a Borel measure on $D$ satisfying condition $(2)$, and write $\nu_j={\bf 1}_{D_j}\nu$, where $(D_j)$ is a sequence of relatively compact convex subsets of $D$, increasing to $D$. Let $\hat{\nu}_j$ be the unique Toric measure in $\Omega$ such that $(\tau_{\zeta})_*(\hat {\nu}_j)=\nu_j$, for all $\zeta\in \mathbb T^n$. Then $\hat{\nu}_j$ vanishes on pluripolar sets and $\hat{\nu}_j(\Omega)<+\infty$. By \cite[Theorem 5.14]{Ceg04}, there exists a unique $\hat{v}_j\in \mathcal F^a(\Omega)$ such that $(dd^c \hat{v}_j)^n = \hat{\nu}_j$. By uniqueness, $\hat{v}_j=v_j\circ L$ for some $v_j\in {\rm CV}(D)$. By \cite[Proposition 7.3]{LZ25}, we have $v_j \in {\rm CV}_0(D)$, hence $\tilde{v}_j\in \mathcal E^0(\Omega)$.  It thus follows from Lemma 
 \ref{continuity I mu}, Lemma \ref{lem:existence minimizer} and Theorem \ref{thm: minimizer is a solution},  that  for each $j\in \N$, there exists $\hat{u}_j \in \mathcal{E}^1_{\phi}(\Omega)$ such that  $(dd^c \hat{u}_j)^n = \hat{\nu}_j$. By uniqueness $\hat{u}_j=u_j\circ L$. The remaining part of this proof is the same as in Theorem \ref{thmA}.

Let us now give an alternative proof of $(2) \Rightarrow (3)$. We will check that $\mathcal E^1_{\phi}(\Omega)\subset \phi+L^1(\Omega,\mu)$, where $\mu=\hat{\nu}$. Take $u\in \mathcal E^1_{\phi}(\Omega)$ and let $u^{\tau}$ be its torification function as defined in \eqref{eq: torification}. By Lemma \ref{lem: torification in E1} we have $u^{\tau} \in \mathcal E^1_{\phi}(\Omega)$. Since $\mu$ and $\phi$ are $\tau_{\zeta}$-invariant, using Fubini's theorem we can write
\begin{flalign*}
	\int_{\Omega} (\phi-u) d\mu &= \int_{\mathbb T^n}\int_{\Omega} (\phi-u)(\tau_{\zeta}(z))\;  d\mu(z) d\sigma_{\mathbb T^n}(\zeta)\\
	&= \int_{\Omega}\int_{\mathbb T^n} (\phi-u)(\tau_{\zeta}(z))\;   d\sigma_{\mathbb T^n}(\zeta) d\mu(z) \\
	&= \int_{\Omega} (\phi-u^{\tau})(z) d\mu(z)\\
	& = \int_D (\varphi-v) {\mathrm M}_{\mathbb R}(v), 
\end{flalign*}
where $u^{\tau}=v\circ L$ with $v\in \mathcal E^1_{\varphi}(D)$. 
The last integral is finite by assumption $(2)$. Hence $\mu$ satisfies $(1)$ in Theorem \ref{thmA}. There thus exists $\hat{w}\in \mathcal E^1_{\phi}(\Omega)$ such that $(dd^c \hat{w})^n =\hat{\nu}$. By uniqueness the solution $\hat{w}$ is then toric and $\hat{w} = w \circ L$ in $\Omega$. By Proposition \ref{prop: energy MAR-MAC} we have $w\in \mathcal E^1_{\varphi}(D)$ and this completes the proof. 
\end{proof}

%\bibliographystyle{alpha}
%\bibliography{//Users/chinhlu/Desktop/Bib/biblio}

\begin{thebibliography}{{\AA}CK{\etalchar{+}}09}

\bibitem[{\AA}C07]{AC07}
Per {\AA}hag and Rafa{\l} Czy{\.z}.
\newblock On the {Cegrell} classes.
\newblock {\em Math. Z.}, 256(2):243--264, 2007.

\bibitem[{\AA}CC12]{ACC12}
Per {\AA}hag, Urban Cegrell, and Rafa{\l} Czy{\.z}.
\newblock On {Dirichlet}'s principle and problem.
\newblock {\em Math. Scand.}, 110(2):235--250, 2012.

\bibitem[{\AA}CCP09]{ACCP09}
Per {\AA}hag, Urban Cegrell, Rafa{\l} Czy{\.z}, and Ho{\`a}ng~Hi{\d{\ecircumflex}}p Pha{\d{m}}.
\newblock Monge-{Amp{\`e}re} measures on pluripolar sets.
\newblock {\em J. Math. Pures Appl. (9)}, 92(6):613--627, 2009.

\bibitem[{\AA}CK{\etalchar{+}}09]{ACKPZ09}
Per {\AA}hag, Urban Cegrell, S{\l}awomir Ko{\l}odziej, Hoang~Hiep Pha{\d{m}}, and Ahmed Zeriahi.
\newblock Partial pluricomplex energy and integrability exponents of plurisubharmonic functions.
\newblock {\em Adv. Math.}, 222(6):2036--2058, 2009.

\bibitem[{\AA}CLR25]{ACLR25}
Per {\AA}hag, Rafa{\l} Czy{\.z}, Chinh~H. Lu, and Alexander Rashkovskii.
\newblock Geodesic connectivity and rooftop envelopes in the {Cegrell} classes.
\newblock {\em Math. Ann.}, 391(3):3333--3361, 2025.

\bibitem[AZ67]{AZ67}
A.~D. Aleksandrov and V.~A. Zalgaller.
\newblock {\em Intrinsic geometry of surfaces}, volume~15 of {\em Transl. Math. Monogr.}
\newblock American Mathematical Society (AMS), Providence, RI, 1967.

\bibitem[BB13]{BB13}
Robert~J. Berman and Bo~Berndtsson.
\newblock Real {Monge}-{Amp{\`e}re} equations and {K{\"a}hler}-{Ricci} solitons on toric log {Fano} varieties.
\newblock {\em Ann. Fac. Sci. Toulouse, Math. (6)}, 22(4):649--711, 2013.

\bibitem[BB22]{BB22}
Robert~J. Berman and Bo~Berndtsson.
\newblock Moser-{T}rudinger type inequalities for complex {M}onge-{A}mp\`ere operators and {A}ubin's ``hypoth\`ese fondamentale''.
\newblock {\em Ann. Fac. Sci. Toulouse Math. (6)}, 31(3):595--645, 2022.

\bibitem[BBGZ13]{BBGZ13}
Robert~J. Berman, S{\'e}bastien Boucksom, Vincent Guedj, and Ahmed Zeriahi.
\newblock A variational approach to complex {Monge}-{Amp{\`e}re} equations.
\newblock {\em Publ. Math., Inst. Hautes {\'E}tud. Sci.}, 117:179--245, 2013.

\bibitem[BGZ09]{BGZ09}
Slimane Benelkourchi, Vincent Guedj, and Ahmed Zeriahi.
\newblock Plurisubharmonic functions with weak singularities.
\newblock In {\em Complex analysis and digital geometry. Proceedings from the Kiselmanfest, Uppsala, Sweden, May 2006 on the occasion of Christer Kiselman's retirement}, pages 57--74. Uppsala: Univ. Uppsala, 2009.

\bibitem[B{\l}o93]{Bl93}
Zbigniew B{\l}ocki.
\newblock Estimates for the complex {Monge}-{Amp{\`e}re} operator.
\newblock {\em Bull. Pol. Acad. Sci., Math.}, 41(2):151--157, 1993.

\bibitem[BT76]{BT76}
Eric Bedford and B.~A. Taylor.
\newblock The {Dirichlet} problem for a complex {Monge}-{Amp{\`e}re} equation.
\newblock {\em Invent. Math.}, 37:1--44, 1976.

\bibitem[BT82]{BT82}
Eric Bedford and B.~A. Taylor.
\newblock A new capacity for plurisubharmonic functions.
\newblock {\em Acta Math.}, 149(1-2):1--40, 1982.

\bibitem[Ceg84]{Ceg84}
Urban Cegrell.
\newblock On the {Dirichlet} problem for the complex {Monge}-{Amp{\`e}re} operator.
\newblock {\em Math. Z.}, 185:247--251, 1984.

\bibitem[Ceg98]{Ceg98}
Urban Cegrell.
\newblock Pluricomplex energy.
\newblock {\em Acta Math.}, 180(2):187--217, 1998.

\bibitem[Ceg04]{Ceg04}
Urban Cegrell.
\newblock The general definition of the complex {Monge}-{Amp{\`e}re} operator.
\newblock {\em Ann. Inst. Fourier}, 54(1):163--184, 2004.

\bibitem[Ceg08]{Ceg08}
Urban Cegrell.
\newblock A general {D}irichlet problem for the complex {M}onge-{A}mp\`ere operator.
\newblock {\em Ann. Polon. Math.}, 94(2):131--147, 2008.

\bibitem[CGSZ19]{CGSZ19}
Dan Coman, Vincent Guedj, Sibel Sahin, and Ahmed Zeriahi.
\newblock Toric pluripotential theory.
\newblock {\em Ann. Pol. Math.}, 123:215--242, 2019.

\bibitem[Fig17]{Fig17}
Alessio Figalli.
\newblock {\em The {Monge}-{Amp{\`e}re} equation and its applications}.
\newblock Zur. Lect. Adv. Math. Z{\"u}rich: European Mathematical Society (EMS), 2017.

\bibitem[GLZ19]{GLZ19}
Vincent Guedj, Chinh~H. Lu, and Ahmed Zeriahi.
\newblock Plurisubharmonic envelopes and supersolutions.
\newblock {\em J. Differ. Geom.}, 113(2):273--313, 2019.

\bibitem[Gut16]{Gut16}
Cristian~E. Guti{\'e}rrez.
\newblock {\em The {Monge}-{Amp{\`e}re} equation}, volume~89 of {\em Prog. Nonlinear Differ. Equ. Appl.}
\newblock Basel: Birkh{\"a}user/Springer, 2nd edition edition, 2016.

\bibitem[GZ17]{GZbook}
Vincent Guedj and Ahmed Zeriahi.
\newblock {\em Degenerate complex {M}onge-{A}mp\`ere equations}, volume~26 of {\em EMS Tracts in Mathematics}.
\newblock European Mathematical Society (EMS), Z\"{u}rich, 2017.

\bibitem[Har06]{Har06}
David Hartenstine.
\newblock The {Dirichlet} problem for the {Monge}-{Amp{\`e}re} equation in convex (but not strictly convex) domains.
\newblock {\em Electron. J. Differ. Equ.}, 2006:9, 2006.
\newblock Id/No 138.

\bibitem[Jos78]{Jos78}
Bengt Josefson.
\newblock On the equivalence between locally polar and globally polar sets for plurisubharmonic functions on {{\(C^n\)}}.
\newblock {\em Ark. Mat.}, 16:109--115, 1978.

\bibitem[Ko{\l}98]{Kol98}
S{\l}awomir Ko{\l}odziej.
\newblock The complex {M}onge-{A}mp\`ere equation.
\newblock {\em Acta Math.}, 180(1):69--117, 1998.

\bibitem[KR81]{KR81}
Norberto Kerzman and Jean-Pierre Rosay.
\newblock Fonctions plurisousharmoniques d'exhaustion bornees et domaines taut.
\newblock {\em Math. Ann.}, 257:171--184, 1981.

\bibitem[Le26]{Le26}
Nam~Q. Le.
\newblock A {Variational} {Approach} to {Degenerate} {Monge}--{Amp{\`e}re} {Equations} with {Mixed} {Measures} and {Monotonicity}.
\newblock Preprint, {arXiv}:2603.19114 [math.{AP}] (2026), 2026.

\bibitem[Lu15]{Lu15}
Chinh~H. Lu.
\newblock A variational approach to complex {Hessian} equations in {{\(\mathbb{C}^n\)}}.
\newblock {\em J. Math. Anal. Appl.}, 431(1):228--259, 2015.

\bibitem[LZ25]{LZ25}
Chinh~H. Lu and Ahmed Zeriahi.
\newblock A new approach to the {Monge}-{Amp{\`e}re} eigenvalue problem.
\newblock Preprint, {arXiv}:2507.18409, 2025.

\bibitem[NP09]{NP09}
Van~Khue Nguyen and Hoang~Hiep Pham.
\newblock A comparison principle for the complex {Monge}-{Amp{\`e}re} operator in {Cegrell}'s classes and applications.
\newblock {\em Trans. Am. Math. Soc.}, 361(10):5539--5554, 2009.

\bibitem[Ran95]{Ran95}
Thomas Ransford.
\newblock {\em Potential theory in the complex plane}, volume~28.
\newblock London Mathematical Society Student Texts, Cambridge University Press, 1995.

\bibitem[RT77]{RT77}
Jeffrey Rauch and B.~A. Taylor.
\newblock The {Dirichlet} problem for the multidimensional {Monge}-{Amp{\`e}re} equation.
\newblock {\em Rocky Mt. J. Math.}, 7:345--364, 1977.

\end{thebibliography}

\end{document}